\documentclass{amsart}
\usepackage{amssymb}
\usepackage{hyperref}

\DeclareMathOperator{\dir}{dir}
\DeclareMathOperator{\spn}{span}
\DeclareMathOperator{\dec}{dec}
\newtheorem{theorem}{Theorem}[section]
\newtheorem{lemma}[theorem]{Lemma}
\newtheorem{proposition}[theorem]{Proposition}
\newtheorem{corollary}[theorem]{Corollary}
\theoremstyle{remark}\newtheorem{remark}[theorem]{Remark}
\numberwithin{equation}{section}

\hypersetup{pdfstartview=FitBH}

\DeclareFontFamily{U}{mathx}{\hyphenchar\font45}
\DeclareFontShape{U}{mathx}{m}{n}{
      <5> <6> <7> <8> <9> <10>
      <10.95> <12> <14.4> <17.28> <20.74> <24.88>
      mathx10
      }{}
\DeclareSymbolFont{mathx}{U}{mathx}{m}{n}
\DeclareFontSubstitution{U}{mathx}{m}{n}
\DeclareMathAccent{\widecheck}{0}{mathx}{"71}
\DeclareMathAccent{\wideparen}{0}{mathx}{"75}

\title{An $L^{4/3}$ $SL_2$ Kakeya maximal inequality}
\author{Terence L.~J.~Harris}
\address{Department of Mathematics, University of Wisconsin, 480
Lincoln Drive, Madison, WI, 53706, USA}
\email{terry.harris@math.wisc.edu}

\begin{document} 

\begin{abstract} It is shown that $SL_2$ Besicovitch sets of measure zero exist in $\mathbb{R}^3$. The proof is constructive and uses point-line duality analogously to Kahane's construction of measure zero Besicovitch sets in the plane. A corollary is that the $SL_2$ Kakeya maximal inequality cannot hold with uniform constant. A counterexample is given to show that the $SL_2$ Kakeya maximal inequality cannot hold for $p> 3/2$; even in the model case where the $\delta$-tubes have $\delta$-separated directions and the cardinality of the tube family is $\sim \delta^{-2}$. It is then shown that, with $C_{\epsilon} \delta^{-\epsilon}$ loss, the $SL_2$ Kakeya maximal inequality does hold if $p \leq 4/3$, whenever the tubes satisfy a 2-dimensional ball condition (equivalent to the Wolff axioms in the $SL_2$ case). The proof is via an $L^{4/3}$ inequality for restricted families of projections onto planes. For both inequalities, the range $4/3 < p \leq 3/2$ remains an open problem. A related $L^{6/5}$ inequality is derived for restricted projections onto lines. Finally, an application is given to generic intersections of sets in $\mathbb{R}^3$ with ``light rays'' and ``light planes''. 
\end{abstract}

\maketitle

\section{Introduction} Given $(x_1,y_1,t_1), (x_2,y_2,t_2) \in \mathbb{R}^3$, let 
\[ (x_1,y_1,t_1) \ast (x_2,y_2,t_2) = \left(x_1+x_2, y_1+y_2, t_1 + t_2 + \frac{1}{2} \left( x_1 y_2 - x_2 y_1 \right) \right), \]
be the Heisenberg product. A line $\ell \subseteq \mathbb{R}^3$ is called horizontal if there exists $\theta \in [0, \pi)$ and $p \in \mathbb{R}^3$ such that 
\[ \ell = p \ast \mathbb{V}_{\theta}, \]
where $\mathbb{V}_{\theta} = \left\{ (\lambda e^{i \theta}, 0 ) :  \lambda \in \mathbb{R} \right\}$. A line $\ell \subseteq \mathbb{R}^3$ is called an $SL_2$ line if there exist $a,b,c,d \in \mathbb{R}$ with $ad-bc = 1$ such that
\[ \ell = \{(a,b,0) + \lambda (c,d,1) : \lambda \in \mathbb{R} \}, \]
or if there exists $t_0 \in \mathbb{R}$ and $(c,d) \in \mathbb{R}^2 \setminus \{0\}$ such that 
\[ \ell = \{(0,0,t_0) + \lambda (c,d,0) : \lambda \in \mathbb{R} \}. \]
Horizontal line segments and $SL_2$ line segments are line segments which are subsets of horizontal and $SL_2$ lines, respectively. A set $K \subseteq \mathbb{R}^3$ will be called a horizontal Besicovitch set if $K$ contains a horizontal unit line segment in every direction in $S^2 \setminus \{(0,0,\pm 1)\}$. Similarly, $K$ will be called an $SL_2$ Besicovitch set if $K$ contains an $SL_2$ unit line segment in every direction in $S^2 \setminus \{(0,0,\pm 1)\}$. If $F(x,y,z) = (x,y,2z)$, then $F$ sends horizontal lines to $SL_2$ lines, and $\ell \mapsto F(\ell)$ is a bijection between horizontal and $SL_2$ lines with inverse $\ell \mapsto F^{-1}(\ell)$, where $F^{-1}(x,y,z) = (x,y,z/2)$. For any horizontal line segment $I$, the length of $F(I)$ is comparable to the length of $I$, so constructing horizontal Besicovitch sets of measure zero is equivalent to constructing $SL_2$ Besicovitch sets of measure zero. 

In Section~\ref{horizontalbesicovitch} it is shown via an explicit construction that there exist closed horizontal Besicovitch sets of measure zero, using the same point-line duality as that used in \cite{liu}. This is then used to show the corollary that a uniform $L^p$ bound (i.e.~without some loss such as $\lvert \log \delta \rvert$) is not possible in the $SL_2$ Kakeya maximal inequality (for nontrivial $p$). In Section~\ref{planktrain}, some examples are given to show that the range of $p$ for which the (dual) $SL_2$ Kakeya maximal inequality can hold on $L^p$ (with $C_{\epsilon} \delta^{-\epsilon}$ loss) is at best $p \leq 3/2$, which is the conjectured range of $p$ for the standard Kakaya maximal function in $\mathbb{R}^3$. This is a counterexample to \cite[Remark~1]{katzwuzahl}, where it was suggested that $p=2$ might be possible in the $SL_2$ case. The version of the $SL_2$ Kakeya maximal inequality in \cite{katzwuzahl} appears slightly different to that used here, but the equivalence of the various forms is discussed at the end of the introduction. In Section~\ref{Lpbounds}, the following $SL_2$ Kakeya maximal inequality is proved.
\begin{theorem} \label{sl2maximal} Let $\mathbb{T}$ be a collection of $SL_2$ tubes in $\mathbb{R}^3$ of length 1 and radius $\delta$, satisfying the condition
\[ |\{ T \in \mathbb{T} : T \subseteq S \} | \leq (r/ \delta)^2, \]
for any tube $S$ of length 2 and radius $r$, and any $r \in [\delta, 1]$. Then for any $\epsilon >0$,
\begin{equation} \label{dualkakeya} \left\lVert \sum_{T \in \mathbb{T} } \chi_T \right\rVert_{L^{4/3}(B_3(0,1) ) }^{4/3} \leq C_{\epsilon} \delta^{-\epsilon} \left\lvert \mathbb{T}\right\rvert \delta^{2}. \end{equation} \end{theorem}
Theorem~\ref{sl2maximal} will be deduced from an $L^{4/3}$ inequality for restricted families of projections onto planes stated below, and the point-line duality principle from \cite{liu,fasslerorponen}.
	\begin{theorem} \label{projection43} If $\alpha >2$ then for any Borel measure $\mu$ on $B_3(0,1)$,
	\[ \int_0^{2\pi} \| \pi_{\theta \#} \mu \|_{L^{4/3}}^{4/3} \, d\theta \leq C_{\alpha} \mu(\mathbb{R}^3) c_{\alpha}(\mu)^{1/3}. \] \end{theorem} 
Here $\pi_{\theta}$ is the orthogonal projection onto $\frac{1}{\sqrt{2}}(\cos \theta, \sin \theta, 1)^{\perp}$, and 
\[ c_{\alpha}(\mu)=\sup_{r>0, x \in \mathbb{R}^3} \frac{ \mu(B(x,r) ) }{r^{\alpha} }. \]
As a counterpart to Theorem~\ref{projection43}, the following version for projections $\rho_{\theta}$ onto the span of $\frac{1}{\sqrt{2}} \left( \cos \theta, \sin \theta, 1 \right)$ will be proved using a similar method.
\begin{theorem} \label{projection44} If $\alpha >1$ then for any Borel measure $\mu$ on $B_3(0,1)$,
	\[ \int_0^{2\pi} \| \rho_{\theta \#} \mu \|_{L^{6/5}}^{6/5} \, d\theta \leq C_{\alpha} \mu(\mathbb{R}^3) c_{\alpha}(\mu)^{1/5}. \] \end{theorem} 
	
	The proofs of Theorem~\ref{projection43} and Theorem~\ref{projection44} here can easily be generalised to cover $C^2$ curves $\gamma: [0,1] \to S^2$ with $\det(\gamma, \gamma', \gamma'')$ nonvanishing, but for simplicity they are only written here for the model curve $\gamma(\theta) = \frac{1}{\sqrt{2}} (\cos \theta, \sin \theta, 1)$. 

A different but similar looking maximal function to the $SL_2$ Kakeya maximal function was considered in \cite{fasslerpinamontiwald}, where the optimal $L^{3/2}$ inequality was obtained. This maximal function was a function of the angle $\theta \in [0,\pi)$ of the tube, rather than the direction. The tubes considered were also ``Heisenberg tubes'' rather than Euclidean tubes, and the problem was related to the projections $\rho_{\theta}$ rather than $\pi_{\theta}$. 

By a recent result of Mattila \cite[Theorem~1.2]{mattilaLp}, a consequence of Theorem~\ref{projection43} and Theorem~\ref{projection44} is the following result about generic intersections of sets in $\mathbb{R}^3$ with ``light rays'' and ``light planes''.
\begin{theorem} \label{intersectiontheorem} Let $A \subseteq \mathbb{R}^3$ be $\mathcal{H}^s$-measurable with $0 < \mathcal{H}^s(A) < \infty$. If $s >1$ then for $\left(\mathcal{H}^s \times \mathcal{H}^1\right)$-a.e.~$(x, \theta) \in A \times [0, 2\pi)$, 
\[ \dim\left( A \cap \rho_{\theta}^{-1}( \rho_{\theta}(x) ) \right) = s-1, \]
and for a.e.~$\theta \in [0, 2\pi)$, 
\[ \mathcal{H}^1\left\{ w \in \spn(\gamma(\theta) ) :  \dim\left( A \cap \rho_{\theta}^{-1}(w) \right) = s-1 \right\} > 0. \]
If $s >2$ then for $\left(\mathcal{H}^s \times \mathcal{H}^1\right)$-a.e.~$(x, \theta) \in A \times [0, 2\pi)$, 
\[ \dim\left( A \cap \pi_{\theta}^{-1}( \pi_{\theta}(x) ) \right) = s-2, \]
and for a.e.~$\theta \in [0, 2\pi)$, 
\[ \mathcal{H}^2\left\{ w \in \gamma(\theta)^{\perp} :  \dim\left( A \cap \pi_{\theta}^{-1}(w) \right) = s-2 \right\} > 0. \]\end{theorem}

Theorem~\ref{intersectiontheorem} is a refinement of the Marstrand-Mattila slicing theorem in $\mathbb{R}^3$ (see Theorem~10.10 and Theorem~10.11 in \cite{mattilageom}), which has the same statement as the above, but with lines and planes varying over all lines and planes through a point, rather than just ``light rays'' and ``light planes''. By Fubini's theorem and a simple scaling argument, Theorem~\ref{intersectiontheorem} implies the standard Marstrand-Mattila slicing theorem in $\mathbb{R}^3$. 

\begin{sloppypar} Theorem~\ref{intersectiontheorem} also implies that if $A \subseteq \mathbb{R}^3$ is Borel with $\dim A >1$, then $\mathcal{H}^1( \rho_{\theta} (A) ) >0$ for a.e.~$\theta \in [0, 2\pi)$, which was first proved in \cite{lengthprojections}. Similarly, Theorem~\ref{intersectiontheorem} implies that if $A \subseteq \mathbb{R}^3$ is Borel with $\dim A >2$, then $\mathcal{H}^2( \pi_{\theta} (A) ) >0$ for a.e.~$\theta \in [0, 2\pi)$, which was first proved in \cite{GGGHMW}. These length and area statements in turn imply the Hausdorff dimension versions $\dim( \rho_{\theta}(A)) = \min\{1, \dim A\}$ and $\dim( \pi_{\theta}(A)) = \min\{2, \dim A\}$ for a.e.~$\theta \in [0, 2\pi)$; Fässler and Orponen proved that for general families of Euclidean projections, the dimension exponents in projection theorems can always be lowered by randomly adding points to lower dimensional sets to obtain higher dimensional sets (see e.g.~\cite{fasslerorponen} for this type of argument), though this method probably loses information about the exceptional set. The a.e.~equality $\dim( \rho_{\theta}(A)) =\min\{1, \dim A\}$ was first obtained by Käenmäki, Orponen and Venieri \cite{KOV}, and then by Pramanik, Yang and Zahl \cite{PYZ} and independently by Gan, Guth and Maldague \cite{GGM}, both for more general curves (and all of \cite{KOV,PYZ,GGM} contain better information about the exceptional set than what follows from Theorem~\ref{intersectiontheorem}). The Hausdorff dimension theorem for the projections $\pi_{\theta}$ was first obtained in \cite{GGGHMW} (and later in \cite{katzwuzahl}). The a.e.~equalities $\dim( \rho_{\theta}(A)) = \min\{1, \dim A\}$ and $\dim( \pi_{\theta}(A)) = \min\{2, \dim A\}$ can also be obtained as corollaries of Theorem~\ref{projection44} and Theorem~\ref{projection43} directly (by Hölder's inequality and the definition of Hausdorff dimension). \end{sloppypar}

The $SL_2$ example related to the Kakeya problem first appeared in Katz and Zahl's $5/2+\epsilon$ lower bound for the Hausdorff dimension of Besicovitch sets in $\mathbb{R}^3$ \cite{KZ}, where the proof needed to negotiate a hypothetical $SL_2$ ``almost-counterexample''. Later, Wang and Zahl conjectured that $SL_2$ Besicovitch sets in $\mathbb{R}^3$ have Hausdorff dimension 3 (as a special case of the Kakeya conjecture in $\mathbb{R}^3$) \cite{WZ}. This was proved by Fässler and Orponen \cite{fasslerorponen2}, and also by Katz, Wu, and Zahl \cite{katzwuzahl}. By a standard argument (see e.g.~\cite[Proposition~10.2]{wolffbook}), Theorem~\ref{sl2maximal} also implies that $SL_2$ Besicovitch sets have Hausdorff dimension 3 (yielding a third proof of this). A type of $SL_2$ Kakeya maximal inequality first appeared in \cite{katzwuzahl}. The inequality there implies that $SL_2$ Besicovitch sets have Hausdorff dimension 3, but does not imply the $SL_2$ maximal inequality in \eqref{dualkakeya} for any $p>1$ (in place of $4/3$). 

Finally, see \cite[Sections~6.7 and 6.8]{mattilasurvey} for a general discussion of the connections between restricted projection problems and Kakeya-type problems. 

\subsection*{Different forms of the \texorpdfstring{$SL_2$}{SL2} Kakeya maximal inequality} \begin{sloppypar} Given a Borel function $f$ on $\mathbb{R}^3$, let $f_{\delta, SL_2}^{*}: S^2 \setminus \{(0,0, \pm 1) \}$ be the $SL_2$ Kakeya maximal function at scale $\delta$, which is defined as for the standard Kakeya maximal function except that the centre lines $\ell(T)$ of the $\delta \times \delta \times 1$-tubes $T$ are required to be $SL_2$ lines:
\[ f_{\delta,SL_2}^*(e) = \sup_{\substack{\dir(T) = e \\
T = \text{$\delta \times \delta \times 1$-tube}\\
\ell(T) \in SL_2}} \frac{1}{m(T)} \int_T |f|, \]
where $m(T)$ is the Lebesgue measure of $T$. The terms ``$SL_2$'' and ``horizontal'' will be used interchangeably, though strictly speaking they are slightly different. \end{sloppypar}

 If $p^*  \in (1,\infty)$ is fixed, then the inequality 
\begin{equation} \label{sl2kakeya} \|f_{\delta, SL_2}^{*}\|_{L^{p^*}(S^2)} \leq C_{\epsilon} \delta^{-\epsilon} \|f\|_{p^*} \end{equation}
is equivalent (by the standard duality argument; see e.g.~\cite[Lemma~10.4]{wolffbook}) to the inequality
\begin{equation} \label{strongmaximal} \left\lVert  \sum_{T \in \mathbb{T} } \chi_T \right\rVert_{p} ^p \leq C_{\epsilon} \delta^{-\epsilon} |\mathbb{T}| \delta^2 \end{equation}
whenever $\mathbb{T}$ is a family of $1 \times \delta \times \delta$ $SL_2$ tubes with $\delta$-separated directions. It is known that the inequality \eqref{strongmaximal} is equivalent to
\begin{equation} \label{shadingversion} \mathcal{H}^3\left( \bigcup_{T \in \mathbb{T} } Y(T) \right) \geq C_{\epsilon}^{-1} \delta^{\epsilon} \lambda^{p^*} \delta^{2} |\mathbb{T}|, \end{equation}
whenever $\mathbb{T}$ is a family of $1 \times \delta \times \delta$ $SL_2$ tubes with $\delta$-separated directions, $\delta \leq \lambda \leq 1$, and $Y$ is a $\lambda$-shading of $\mathbb{T}$ (i.e.~a union of $\delta$-cubes such that at least a fraction $\lambda$ of each tube in $\mathbb{T}$ is covered by the cubes in $Y$, meaning that $\mathcal{H}^3(Y(T)) \geq \lambda \mathcal{H}^3(T)$ for any $T \in \mathbb{T}$, where $Y(T) = Y \cap T$). The proof is identical to that for the standard Kakeya maximal function, but will be summarised here for readability. That \eqref{strongmaximal} implies \eqref{shadingversion} follows from Hölder's inequality. The reverse implication follows by pigeonholing a set $Y$ of $\delta$-cubes on which the value of $\sum_{T \in \mathbb{T} } \chi_T$ is approximately constant in order to reverse Hölder's inequality, then by pigeonholing again to a subset $\mathbb{T}' \subseteq \mathbb{T}$ (of potentially smaller cardinality) of tubes containing an approximately equal fraction $\lambda$ of the cubes from $Y$ in each tube, and then applying \eqref{shadingversion} to this subset. 

An a priori weaker inequality than \eqref{strongmaximal} is
\begin{equation} \label{weakmaximal} \left\lVert  \sum_{T \in \mathbb{T} } \chi_T \right\rVert_{L^p(B_3(0,1) )} ^p \leq C_{\epsilon} \delta^{-\epsilon}, \end{equation}
whenever $\mathbb{T}$ is a family of $1 \times \delta \times \delta$ $SL_2$ tubes with $\delta$-separated directions. Clearly \eqref{strongmaximal} implies \eqref{weakmaximal}, and the implication can be reversed for the standard Kakeya maximal function by rotational averaging (see e.g.~\cite[Proposition~22.7]{mattila}), but rotations in $O(3)$ do not necessarily preserve horizontal lines, so it does not seem obvious that \eqref{weakmaximal} implies \eqref{strongmaximal} in the $SL_2$ case. In Section~\ref{planktrain} it is shown that \eqref{weakmaximal} cannot hold in the $SL_2$ case unless $p \leq 3/2$ (and therefore neither can \eqref{shadingversion}, \eqref{strongmaximal} or \eqref{sl2kakeya}). 

\section*{Acknowledgements} The ``microlocal interpolation'' argument from \cite[Exercise~9.21]{demeter} used in Section~\ref{Lpbounds} was explained to me in a different context by Alex Barron. A possible connection between the $L^p$ $SL_2$ Kakeya maximal inequality and $L^p$ bounds for restricted families of projections was suggested to me by Josh Zahl.

This work was partially supported by funding from the American Institute of Mathematics (AIM).

\section{\texorpdfstring{$SL_2$}{SL2} Besicovitch sets of measure zero}

\label{horizontalbesicovitch} 

Every horizontal line can be written in the form $p \ast \mathbb{V}_{\theta}$ with $(p_1,p_2)$ in the span of $i e^{i \theta}$, where $p = (p_1,p_2,p_3)$. If $R>0$ and if such a horizontal line intersects $B(0,R)$, then $\lvert (p_1, p_2) \rvert \leq R$ and the direction of $p \ast \mathbb{V}_{\theta}$ has distance at least $2/R$ from $(0,0, \pm 1)$. It follows that a bounded set cannot contain a horizontal unit line segment in every direction in $S^2 \setminus \{0,0, \pm 1\}$ (and similarly for $SL_2$ line segments). For this reason, the compact sets of horizontal and $SL_2$ line segments in Theorem~\ref{besicovitchhorizontal} below contain only unit line segments in every direction in $S^2  \setminus \mathcal{N}_{\epsilon}(\{(0,0,\pm 1)\})$ for a given $\epsilon >0$, and to get all directions in $S^2  \setminus \{(0,0,\pm 1)\}$ it is necessary to replace ``compact'' with ``closed''. 

The following (straightforward) lemma is used to make the proof of Theorem~\ref{besicovitchhorizontal} below constructive, though it is possible to give a non-constructive proof of Theorem~\ref{besicovitchhorizontal} below by substituting the application of Lemma~\ref{lemmaconstr} in the proof of Theorem~\ref{besicovitchhorizontal} with an averaging/Fubini/scaling argument. 

\begin{sloppypar} \begin{lemma} \label{lemmaconstr} If $E \subseteq (-\pi/2, \pi/2)$ is Lebesgue measurable with $\mathcal{H}^1(E) = 0$, and if $F$ is the union of great circles defined by
\[ F = \{ ( \lambda e^{i \theta}, \mu) \in S^2 : (\lambda, \mu) \in S^1, \quad \theta \in E \}, \]
 then $\mathcal{H}^1(F \cap C) = 0$ for any circle $C \subseteq S^2$ with the property that either $(0,0,1) \notin C$ or $(0,0,-1) \notin C$. 
\end{lemma}
\begin{proof} Let $E, C$ be given. By stereographic projection on the Riemann sphere (which maps circles to circles or lines) it suffices to show that if $\widetilde{C}$ is either a circle in $\mathbb{R}^2$ or a line in $\mathbb{R}^2$ which does not pass through the origin, then $\mathcal{H}^1\left( \widetilde{F} \cap \widetilde{C}\right) = 0$, where 
\[ \widetilde{F} = \{ \lambda e^{i \theta}: \lambda \in \mathbb{R}, \quad \theta \in E \}. \]
To show this, it suffices to show that for any $(x_0,y_0) \in \widetilde{C} \setminus \{0\}$, there exists an $\epsilon >0$ such that $\mathcal{H}^1\left( \widetilde{F} \cap \widetilde{C} \cap B((x_0,y_0), \epsilon)\right) = 0$. By scaling and rotation it may be assumed that $(x_0,y_0) = (1,0)$. If $\widetilde{C}$ is a circle, then let $(a,b)$ and $r>0$ be such that $\widetilde{C}$ is parametrised by $\gamma(t) = (\gamma_1(t), \gamma_2(t) ) =  (a + r\cos t, b + r \sin t)$, where $\gamma: [0, 2\pi) \to \widetilde{C}$ satisfies $\gamma(t_0) = (1,0)$ for some $t_0 \in [0, 2\pi)$. Otherwise let $\gamma(t) = (1,0) + t(v_1,v_2)$ be the parametrisation of $\widetilde{C}$, where $v_2 \neq 0$ and $t \in \mathbb{R}$, and define $t_0=0$. It suffices to show that for $\delta>0$ sufficiently small,
\[ \mathcal{H}^1\{ t \in (t_0-\delta, t_0+\delta) : \arctan(\gamma_2(t)/\gamma_1(t) ) \in E \} = 0. \]
Equivalently, if $G:(t_0-\delta, t_0+\delta) \to (-\pi/2, \pi/2)$ is defined by  $G(t) = \arctan(\gamma_2(t)/\gamma_1(t) )$, then $(G_{\#}\mathcal{H}^1)(E) = 0$ for $\delta>0$ sufficiently small. Since $\mathcal{H}^1(E) = 0$, it is enough to show that $G_{\#}\mathcal{H}^1 \ll \mathcal{H}^1$ for $\delta>0$ sufficiently small. If $\widetilde{C}$ is a circle and $(a,b) = 0$, then $G'(t) = \frac{d}{dt} \arctan(\gamma_2(t)/\gamma_1(t) )  = 1$ for all $t$ in a neighbourhood of $t =t_0= 0$, and $G_{\#}\mathcal{H}^1 \ll \mathcal{H}^1$ follows. If $\widetilde{C}$ is a line, then 
\[  G'(t) = \frac{v_2}{|\gamma(t)|^2} \neq 0, \]
for all $t$ in a neighbourhood of $t_0=0$, and again $G_{\#}\mathcal{H}^1 \ll \mathcal{H}^1$ follows. The remaining case is where $\widetilde{C}$ is a circle and $(a,b) \neq 0$. Let $H(t) = \det( \gamma(t), \gamma'(t))$. Then
\[ G'(t) = \frac{H(t)}{|\gamma(t)|^2}, \]
for all $t$ in a neighbourhood of $0$, so if $H(t)$ is nonvanishing at $t=t_0$, it follows that $G_{\#}\mathcal{H}^1 \ll \mathcal{H}^1$ for $\delta>0$ sufficiently small. If $H(t_0)=0$, then $H(t) \neq 0$ for all $t \in (t_0-\delta, t_0+\delta) \setminus \{t_0\}$ for some $\delta>0$ sufficiently small, which follows from the identity $|H'(t)|^2 + |H''(t)|^2 = r^2(a^2 + b^2)$. Therefore $G_{\#}\mathcal{H}^1 \ll \mathcal{H}^1$ for $\delta>0$ sufficiently small. This finishes the proof. \end{proof}
\begin{theorem} \label{besicovitchhorizontal}  \hspace{2em}
\begin{enumerate} 
\item For any $\epsilon >0$, there exists a compact set of measure zero containing a horizontal unit line segment in every direction in $S^2 \setminus \mathcal{N}_{\epsilon}(\{(0,0,\pm 1)\})$. 
\item There exists a closed set of measure zero containing a horizontal unit line segment in every direction in $S^2 \setminus \{(0,0,\pm 1)\}$. 
\item There exists a closed set of measure zero containing a horizontal line in every direction in $S^2 \setminus \{(0,0,\pm 1)\}$.
 \end{enumerate} 
The above claims all hold if ``a horizontal'' is replaced by ``an $SL_2$''.\end{theorem}
\begin{proof} Let $A$ be any compact subset of $\mathbb{R}^3$ such that $0 < \mathcal{H}^2(A) < \infty$, such that the projection of $A$ down to the $(x,y)$-plane contains a nonempty open disc, and such that for any non-great circle $C \subseteq S^2$, $\mathcal{H}^2(P_{v^{\perp} }(A)) = 0$ for $\mathcal{H}^1$-a.e.~$v \in C$. Such a set can be constructed as follows. Let $\mathcal{C}$ be the Cantor set in $\mathbb{R}$ of dimension $1/2$ obtained by starting with the unit interval $[0,1]$ and removing the open middle half interval at each step. It is a standard result that
\[ \{ 2x + y : x,y \in \mathcal{C} \} = [0, 3], \]
as can be seen, e.g., by characterising $\mathcal{C}$ as those elements of $[0,1]$ of the form $\sum_{j=1}^{\infty} \frac{\varepsilon_j}{4^j}$ with $\varepsilon_j \in \{0,3\}$ for all $j$.
It follows that if $B = \mathcal{C} \times \mathcal{C} \times [0,1]$, then the projection of $B$ onto the plane spanned by $\frac{1}{\sqrt{5}}(2,1,0)$ and $(0,0,1)$ contains a nonempty open disc. Moreover, if $v \in S^2$ is of the form $(\lambda_1 e^{i \theta}, \lambda_2)$, then 
\[ \pi_{v^{\perp}}(B) \subseteq \pi_{(ie^{i \theta},0)}(\mathcal{C} \times \mathcal{C} \times \{0\}) + \pi_{(\lambda_2 e^{i \theta}, -\lambda_1) }\left(\mathbb{R}^3\right). \]
Hence, by Fubini's theorem,
\[ \left\{v \in S^2 : \mathcal{H}^2(\pi_{v^{\perp}}(B) ) > 0 \right\} \subseteq  F := \left\{ (\lambda_1 e^{i \theta}, \lambda_2) : \theta \in E, \quad (\lambda_1, \lambda_2) \in S^1 \right\}, \]
where 
\[ E = \left\{ \theta \in [0, \pi) : \mathcal{H}^1(\pi_{ie^{i \theta}}(\mathcal{C} \times \mathcal{C}) ) >0 \right\}. \]
But $\mathcal{H}^1(E) = 0$ (see e.g.~\cite[Theorem~10.1]{mattila}), and by Lemma~\ref{lemmaconstr} this implies that $\mathcal{H}^1(F \cap C)=0$ for any non-great circle $C$. It follows that for any non-great circle $C$, $\mathcal{H}^2(\pi_{v^{\perp}}(B) )=0$ for $\mathcal{H}^1$-a.e.~$v \in C$. Since this property is rotation-invariant, replacing $B$ by a rotation of $B$ which sends the plane spanned by $\frac{1}{\sqrt{5}}(2,1,0)$ and $(0,0,1)$ to the $(x,y)$-plane in $\mathbb{R}^3$ yields the required set $A$.

For the first claim in the theorem statement, let 
\[ K = K(A)  = \bigcup_{(a,b,c) \in A} \left\{ \left(as+b, s, c + \frac{bs}{2} \right) : \lvert s \rvert \leq \frac{1}{\sqrt{ 4 + 4a^2 + b^2 }} \right\} . \]
Then, by Fubini's theorem,
\[ \mathcal{H}^3(K) \leq \int_{-1}^1 \mathcal{H}^2\left\{ (x,y) \in \mathbb{R}^2 : (x,y) = \left( as +b , c + \frac{bs}{2} \right) : (a,b,c) \in A \right\} \, ds. \]
For each $s \in (-1,1)$, the set
\[ \left\{ (x,y) \in \mathbb{R}^2 : (x,y) = \left( as +b , c + \frac{bs}{2} \right) : (a,b,c) \in A \right\} \]
can be written as 
\[ \{ \left( \langle p, (s,1,0) \rangle, \langle p, (0,s/2,1) \rangle \right) \in \mathbb{R}^2 : p \in A \}. \]
The vectors $(s,1,0)$ and $(0,s/2,1)\}$ are both orthogonal to $(1,-s,s^2/2)$. The vectors $(s,1,0)$ and $(0,s/2,1)$ are not necessarily orthogonal to each other, but by following the Gram-Schmidt process\footnote{The same observation was made in \cite{fasslerorponen2}.}, the above can be written as 
\[ A_s\left\{ \left( \langle p, (s,1,0) \rangle, \left\langle p, (0,s/2,1) - \frac{s}{2(1+s^2) }(s,1,0) \right\rangle \right) \in \mathbb{R}^2 : p \in A \right\}, \]
where $A_s: \mathbb{R}^2 \to \mathbb{R}^2$ is the linear map
\[ A_s(x,y) = \left(x, y+\frac{sx}{2(1+s^2) }\right). \]
Since $\det A_s = 1$, 
\begin{multline*} \mathcal{H}^2\{ \left( \langle p, (s,1,0) \rangle, \langle p, (0,s/2,1) \rangle \right) \in \mathbb{R}^2 : p \in A  \}  \\
= \mathcal{H}^2\left\{ \left( \left\langle p, (s,1,0) \rangle, \langle p, (0,s/2,1) - \frac{s}{2(1+s^2) }(s,1,0) \right\rangle \right) \in \mathbb{R}^2 : p \in A \right\} \\
\sim \mathcal{H}^2\left( \pi_{\left(1, -s,s^2/2\right)^{\perp}}(A)\right). \end{multline*}
The curve $(1, -s,s^2/2)$ is a parabola inside the cone $\eta_2^2 = 2 \eta_1 \eta_3$, and this cone is a clockwise rotation of the cone $\xi_3^2 = \xi_1^2 + \xi_2^2$ by $\pi/4$ in the $(\xi_1, \xi_3)$-plane. It follows that the normalised curve $\frac{1}{\sqrt{1+s^2 + s^4/4} }(1, -s,s^2/2)$ lies in the intersection of the cone $\eta_2^2 =2\eta_1 \eta_3$ with the sphere $S^2$, and this intersection is a non-great circle in $S^2$. Hence $\mathcal{H}^2\left(\pi_{(1, -s,s^2/2)^{\perp}}(A)\right) = 0$ for a.e.~$s \in \mathbb{R}$, which yields $\mathcal{H}^3(K) = 0$. Since the projection of $A$ down to the $(x,y)$-plane contains an open disc $U \subseteq \mathbb{R}^2$, for all $(a,b) \in U$ the set $K$ contains a horizontal unit line segment with direction $\frac{1}{a^2 + 1 + (b/2)^2}(a,1,b/2)$.

Any translate $A_{(x_0,y_0)} := A + (x_0,y_0,0)$ of $A$ with $(x_0,y_0) \in \mathbb{R}^2$ has the property that the projection of $A_{(x_0,y_0)}$ to $\mathbb{R}^2 \times \{0\}$ contains the open disc $U_{(x_0,y_0)} := U + (x_0,y_0)$, and also has the property that for any non-great circle $C$, the projection of $A_{(x_0,y_0)}$ onto $v^{\perp}$ has $\mathcal{H}^2$-measure zero for $\mathcal{H}^1$-a.e.~$v \in C$. Fix $(a_0,b_0) \in U$. For any $v \in S^2 \setminus \mathbb{R} \times \{0\} \times \mathbb{R}$, choose $(a_v,b_v) \in \mathbb{R}^2$ such that $(a_v+a_0, 1, (b_v+b_0)/2)$ is parallel to $v$. Then $K\left(A_{(a_v,b_v)}\right)$ contains a horizontal unit line segment with direction $\frac{1}{a^2 + 1 + (b/2)^2}(a,1,b/2)$ for all $(a,b) \in U + (a_v,b_v)$. This shows that for any $v \in S^2 \setminus \mathbb{R} \times \{0\} \times \mathbb{R}$, there is an open neighbourhood $U_v \subseteq S^2$ of $v$ and a measure zero compact set $K\left(A_{(a_v,b_v)}\right)$ containing a horizontal unit line segment with every direction in $U_v$. Since the collection of horizontal lines is invariant under rotations of the $(x,y)$-plane (which fix the third coordinate), this implies that for any $v \in S^2 \setminus \{(0,0,\pm 1\}$, there is an open neighbourhood $U_v$ of $v$ and a measure zero compact set containing a horizontal unit line segment with every direction in $U_v$.  This implies that for any $\epsilon >0$, there exists a measure zero compact set containing a horizontal unit line segment in every direction in $S^2 \setminus \mathcal{N}_{\epsilon}( \{(0,0, \pm 1 ) \} )$. This verifies the first claim.

For the second claim, take a countable sequence of translates $A_1,A_2, \dotsc$ of $A$ by vectors in $\mathbb{R}^2 \times \{0\}$ such that $d(A_j,0) \to \infty$, and such that the projections of the $A_j$'s onto $\mathbb{R}^2 \times \{0\}$ cover $\mathbb{R}^2$. Let
\begin{multline*} K_1 = (\mathbb{R} \times \{0\} \times \mathbb{R}) \cup \bigcup_{(a,b,c) \in \bigcup_j A_j } K(A_j) \\
 =  (\mathbb{R} \times \{0\} \times \mathbb{R}) \cup \bigcup_{(a,b,c) \in \bigcup_j A_j } \left\{ \left(as+b, s, c + \frac{bs}{2} \right) : \lvert s \rvert \leq \frac{1}{\sqrt{ 4 + 4a^2 + b^2 }}  \right\}, \end{multline*}
and let $K_2$ be a rotation of $K_1$ by $\pi/2$ in the $(x,y)$-plane. It has already been shown that $K = K_1 \cup K_2$ has measure zero and contains a horizontal unit line segment in every direction in $S^2 \setminus \{(0,0,\pm 1)\}$. By the condition $d(0, A_j) \to \infty$, together with the compactness of $A$ and the definition of $K$, it is straightforward to check that $K$ is closed. This verifies the second claim. 

 The third claim follows from a similar argument with the same $A_j$'s; by taking
\[ K_1 = (\mathbb{R} \times \{0\} \times \mathbb{R}) \cup  \bigcup_{(a,b,c) \in \bigcup_j A_j } \left\{ \left(as+b, s, c + \frac{bs}{2} \right) : s \in \mathbb{R}  \right\}. \]
Then $K = K_1 \cup K_2$ has measure zero by a similar argument to the above, and is a closed set of measure zero containing a horizontal line in every direction in $S^2 \setminus \{(0,0, \pm 1)\}$. 

Finally, the proof for $SL_2$ lines is similar.  \end{proof} \end{sloppypar} 

\begin{corollary} If $p \in (1,\infty]$ and
\begin{equation} \label{constineq} \left\lVert \sum_{T \in \mathbb{T}} \chi_T \right\rVert_{L^p(B(0,1) ) } \leq C_{\delta}, \end{equation}
for any set $\mathbb{T}$ of $\delta \times \delta \times 1$ $SL_2$ tubes with $\delta$-separated directions in $S^2$, then $\lim_{\delta \to 0} C_{\delta} = \infty$.

Let $U$ be any Borel subset of $S^2 \setminus \{(0,0,\pm 1)\}$ with $\mathcal{H}^2(U) >0$. If, for some fixed $p \in [1, \infty)$, 
\[ \left\lVert f_{\delta, SL_2}^{*}\right\rVert_{L^p(U) } \leq C_{\delta,U} \|f\|_p, \]
for all $\delta>0$ and for any non-negative Borel function $f$ on $\mathbb{R}^3$, then $\lim_{\delta \to 0} C_{\delta,U} =\infty$. 
 \end{corollary}
\begin{proof}  For the first part, using Theorem~\ref{besicovitchhorizontal} let $B$ be a compact set of measure zero containing a unit line segment in every direction in $U = S^2 \setminus \mathcal{N}_{0.01}\{(0,0,\pm 1) \}$. Let $\widetilde{U} = \{v_1, \dotsc, v_N\}$ be a maximal $\delta$-separated subset of $U$, where $N \sim \delta^{-2}$. For each $v_k \in \widetilde{U}$, let $T_k$ be a $\delta \times \delta \times 1$ $SL_2$ tube such that $T \subseteq \mathcal{N}_{\delta}(B)$. Then (for a sufficiently large absolute constant $C$)
\[ 1 \lesssim \int_{B(0,C)} \sum_{k=1}^N \chi_{T_k} \leq \mathcal{H}^3( \mathcal{N}_{\delta}(B) )^{1- \frac{1}{p} } \left\lVert \sum_{k=1}^N \chi_{T_k} \right\rVert_{L^p(B(0,C) ) }. \]
Hence, by \eqref{constineq}, 
\[ 1 \lesssim \mathcal{H}^3( \mathcal{N}_{\delta}(B) )^{1- \frac{1}{p} } C_{\delta}. \]
Letting $\delta \to 0$ proves the first part.

For the second part, let $V \subseteq U$ be a compact subset of $U$ such that $\mathcal{H}^2(V)>0$. Let $B$ be a compact set of measure zero containing an $SL_2$ unit line segment in every direction in $V$, which exists by Theorem~\ref{besicovitchhorizontal}. Then 
\[ 1 \lesssim \left\lVert\left( \chi_{\mathcal{N}_{\delta}}(B)\right)_{\delta, SL_2}^{*}\right\rVert_{L^p(V) } \leq  C_{\delta,U} \mathcal{H}^3(\mathcal{N}_{\delta}(B) )^{1/p}. \]
Letting $\delta \to 0$ proves the second part.     \end{proof}

\section{Train tracks of planks} \label{planktrain} 

Given a Borel measure $\mu$ on $\mathbb{R}^3$, and $\alpha \geq 0$, recall that
\[ c_{\alpha}(\mu) = \sup_{x \in \mathbb{R}^3, r >0} \frac{ \mu(B(x,r))}{r^{\alpha}}. \]
Let
\[  I_{\alpha}(\mu) = \int \int \frac{1}{|x-y|^{\alpha}} \, d\mu(x) \, d\mu(y). \]
Recall that $\gamma: [0, 2\pi) \to S^2$ is defined by $\gamma(\theta) = \frac{1}{\sqrt{2}} \left( \cos \theta, \sin \theta, 1 \right)$, and recall that for each $\theta \in [0, 2\pi)$, $\pi_{\theta}$ is the orthogonal projection onto the plane $\gamma(\theta)^{\perp}$. In \cite{oberlin}, Oberlin and Oberlin used Erdoğan's $L^2$ bound for the decay of conical averages of Fourier transforms of fractal measures from \cite{erdogan} to prove a bound on the average $L^2$ norms of pushforwards of fractal measures under $\pi_{\theta}$. When $\alpha > 5/2$, the uncertainty principle suggests that there is no loss in the approach from \cite{oberlin} (up to the endpoint, provided only $L^2$ norms are considered), and this suggests that the Knapp examples from \cite{erdogan} used to prove sharpness of the $L^2$ conical decay rates can also be used to prove sharpness of the bound from \cite{oberlin}. The following proposition verifies this intuition. 
 
\begin{proposition} \label{wavepacket} Let $\alpha \in [2,3]$. If 
\begin{equation} \label{hypothesis} \int_0^{2\pi}  \left\lVert \pi_{\theta \#} \mu \right\rVert_{L^2\left(\gamma(\theta)^{\perp}\right)}^2 \, d\theta \leq C_{\alpha} \mu(\mathbb{R}^3) c_{\alpha}(\mu), \end{equation}
for some nonzero finite $C_{\alpha}$ depending only on $\alpha$, for all Borel measures $\mu$ on $B_3(0,1)$, then $\alpha \geq 5/2$. The same is true if $\mu(\mathbb{R}^3) c_{\alpha}(\mu)$ in \eqref{hypothesis} is replaced by $I_{\alpha}(\mu)$. \end{proposition}
\begin{remark} Oberlin and Oberlin proved that \eqref{hypothesis} does hold if $\alpha > 5/2$, with either $\mu(\mathbb{R}^3) c_{\alpha}(\mu)$ or $I_{\alpha}(\mu)$ on the right-hand side \cite{oberlin}. \end{remark}
\begin{proof}[Proof of Proposition~\ref{wavepacket}] Let $\alpha \in [2,3]$ and suppose that \eqref{hypothesis} holds. Let $\delta>0$ be small. Fix any $\theta_0 \in [0, 2\pi)$. Let $\psi$ be a smooth non-negative bump function supported in $B(0,1)$, with $\psi \sim 1$ on $B(0,1/2)$, such that $\int \psi = 1$. Let $\mu$ be the measure with Radon-Nikodym derivative equal to 
\[ \mu(x) = \delta^{-3/2}\psi\left( \langle x, \gamma(\theta_0) \rangle,  \left\langle \delta^{-1/2}x, \sqrt{2} \gamma'(\theta_0) \right\rangle, \left\langle \delta^{-1} x, \sqrt{2} \left(\gamma \times \gamma'\right)(\theta_0) \right\rangle \right), \]
i.e., $\mu$ is a Schwartz function of $L^1$ norm 1 supported on a plank of dimensions $\sim 1 \times \delta^{1/2} \times \delta$ centred at the origin, such that the longest is direction is parallel to $\gamma(\theta_0)$, the medium direction is parallel to $\gamma'(\theta_0)$ and the shortest direction is parallel to $(\gamma \times \gamma')(\theta_0)$. If $|y_1| \leq \delta^{1/2}/100$, $|y_2| \leq \delta/100$, and $|y_3| \leq 1/100$, then 
\[ \mu\left(y_1 \gamma'(\theta_0) + y_2 \left( \gamma \times \gamma' \right)(\theta_0)  + y_3 \gamma(\theta) \right) \gtrsim \delta^{-3/2}. \]
By second-order Taylor approximation, it follows that if $|\theta-\theta_0| < c\delta^{1/2}$ for $c>0$ a sufficiently small absolute constant, then for all $|x_1| \leq 10^{-3} \delta^{1/2}$ and $|x_2| \leq 10^{-3} \delta$,
\begin{multline*} (\pi_{\theta \#} \mu)\left( x_1 \gamma'(\theta) + x_2 \left( \gamma \times \gamma' \right)(\theta) \right) \\
\geq \int_{-10^{-3}}^{10^{-3} } \mu\left(x_1 \gamma'(\theta) + x_2 \left( \gamma \times \gamma' \right)(\theta)  + t \gamma(\theta) \right) \, dt \gtrsim \delta^{-3/2}. \end{multline*}
Therefore
\begin{align*} &\int_0^{2\pi} \left\lVert \pi_{\theta \#} \mu \right\rVert_{L^2(\gamma(\theta)^{\perp})}^2 \, d\theta \\
&\geq \int_{|\theta - \theta_0| \leq c \delta^{1/2}} \int_{\left[-10^{-3} \delta^{1/2}, 10^{-3} \delta^{1/2} \right]} \int_{\left[-10^{-3} \delta, 10^{-3}\delta\right]} \\
&\qquad \left\lvert (\pi_{\theta \#} \mu)\left( x_1 \gamma'(\theta) + x_2 \left( \gamma \times \gamma' \right)(\theta) \right)\right\rvert^2 \, dx_1 \, dx_2 \, d\theta \\
&\gtrsim \delta^{-1}. \end{align*}
But $\mu(\mathbb{R}^3) =1$, and it is straightforward to check that 
\[ c_{\alpha}(\mu) \sim \delta^{\frac{3}{2}-\alpha}. \]
Similarly $I_{\alpha}(\mu) \sim \delta^{\frac{3}{2}-\alpha}$ (e.g.~via the Plancherel formula for the energy). Hence, by the assumed inequality \eqref{hypothesis},
\[ \delta^{-1} \lesssim \delta^{\frac{3}{2}-\alpha}. \]
Letting $\delta \to 0$ gives $-1 \geq \frac{3}{2}-\alpha$ or $\alpha \geq 5/2$.  \end{proof}

In the previous example, $c_{\alpha}(\mu)$ is much larger than $\mu(\mathbb{R}^3)$. But if $\nu$ is the sum of $N$ copies of $\mu$ translated in the $\gamma'(\theta_0)$ and $(\gamma \times \gamma')(\theta_0)$ directions, such that the projections of these translated copies under $\pi_{\theta}$ are pairwise disjoint for $|\theta-\theta_0| \lesssim \delta^{1/2}$, then the lower bound for the left-hand side will be multiplied by $N$. If the translations are chosen sufficiently sparse, to ensure that $c_{\alpha}(\nu)$ is not larger than $c_{\alpha}(\mu)$, then the right-hand side will also be scaled by $N$, since $\mu(\mathbb{R}^3)$ will be scaled by $N$. If $N$ could be chosen large enough to make $\nu(\mathbb{R}^3)\sim \delta^{3/2-\alpha} \sim c_{\alpha}(\nu) \sim c_{\alpha}(\mu)$, then this would show that no inequality is possible even with the larger right-hand side of $c_{\alpha}(\nu)^2$ instead of $\nu(\mathbb{R}^3)c_{\alpha}(\nu)$ in \eqref{hypothesis}. The proposition below shows that this is possible using ``parallel train tracks of planks''. The spacing is based on the example of ``parallel train tracks'' from \cite[Proposition~6.1]{GIOW}, where the long spaces are a $\delta^{-1/2}$-multiple of the short spaces (the $R$ in \cite{GIOW} corresponds to $\delta^{-1}$ here).

\begin{proposition} \label{traintrackprop} Let $\alpha \in [2,3]$. If 
\begin{equation} \label{traintrack} \int_0^{2\pi} \left\lVert \pi_{\theta \#} \nu \right\rVert_{L^2(\gamma(\theta)^{\perp})}^2 \, d\theta \leq C_{\alpha} c_{\alpha}(\nu)^2, \end{equation}
for some nonzero finite $C_{\alpha}$ depending only on $\alpha$, for all Borel measures $\nu$ on $B_3(0,1)$, then $\alpha \geq 5/2$. \end{proposition}
\begin{proof} Let $\alpha \in [2,3]$ and suppose that \eqref{traintrack} holds. Let $\delta>0$ be small, and fix any $\theta_0 \in [0, 2\pi)$. Let $\mu = \mu_{\delta, \theta_0}$ be the measure from the proof of Proposition~\ref{wavepacket}, given by 
\[ \mu(x) = \delta^{-3/2}\psi\left( \langle x, \gamma(\theta_0) \rangle,  \left\langle \delta^{-1/2}x, \sqrt{2} \gamma'(\theta_0) \right\rangle, \left\langle \delta^{-1} x, \sqrt{2} \left(\gamma \times \gamma'\right)(\theta_0) \right\rangle \right), \]
where $\psi$ is a smooth non-negative bump function supported in $B(0,1)$, with $\psi \sim 1$ on $B(0,1/2)$, such that $\int \psi = 1$. Let $\nu$ be the sum of translated copies of $\mu$; spacing $\delta^{\frac{\alpha}{2}-\frac{1}{2}}$ in the short direction $(\gamma \times \gamma')(\theta_0)$, and spacing $\delta^{\frac{\alpha}{2}-1}$ in the medium direction $\gamma'(\theta_0)$, given by 
\[ \nu(x) = \sum_{|m| \leq 10^{-3} \delta^{\frac{1}{2} - \frac{\alpha}{2}}} \sum_{|n| \leq 10^{-3} \delta^{1 - \frac{\alpha}{2}}} \mu\left( x- \delta^{\frac{\alpha}{2} - \frac{1}{2} } m \left(\gamma \times \gamma'\right)(\theta_0) - n \delta^{\frac{\alpha}{2} -1 } \gamma'(\theta_0) \right). \]
 There are $\sim \delta^{\frac{3}{2}-\alpha}$ such copies, and the supports of the projections of these translated copies under $\pi_{\theta}$ are pairwise disjoint for $|\theta-\theta_0| \leq c\delta^{1/2}$, for a sufficiently small absolute constant $c$ (using $\alpha \leq 3$ and second-order Taylor approximation), so the lower bound of $\delta^{-1}$ from the proof of Proposition~\ref{wavepacket} is multiplied by $\delta^{\frac{3}{2}-\alpha}$ to get 
\[ \int_0^{2\pi} \left\lVert \pi_{\theta \#} \nu \right\rVert_{L^2(\gamma(\theta)^{\perp})}^2 \, d\theta \gtrsim \delta^{\frac{1}{2}-\alpha}. \] 
Moreover, $\nu(\mathbb{R}^3) \sim \delta^{\frac{3}{2}-\alpha}$, and 
\begin{equation} \label{tricky} c_{\alpha}(\nu) \sim  \delta^{\frac{3}{2}-\alpha}. \end{equation}
The lower bound $c_{\alpha}(\nu) \gtrsim  \delta^{\frac{3}{2}-\alpha}$ in \eqref{tricky} follows from $c_{\alpha}(\nu) \gtrsim \nu(\mathbb{R}^3)$. The upper bound can be shown by considering different ranges of $r$ separately, as follows. The maximum of $\nu(B(x,r))/ r^{\alpha}$ over the range $0 < r \leq \delta$ occurs near $r=\delta$ and is $\lesssim \delta^{\frac{3}{2}-\alpha}$. The maximum for the range $\delta \leq r \leq\delta^{\frac{\alpha}{2}-\frac{1}{2}}$ also occurs near $r = \delta$ (since the ball cannot intersect multiple ``slats''), and is $\lesssim \delta^{\frac{3}{2}-\alpha}$. The maximum in the range $\delta^{\frac{\alpha}{2}-\frac{1}{2}} \leq r \leq \delta^{1/2}$ occurs near $r = \delta^{1/2}$, and is $\lesssim \delta^{\frac{3}{2}-\alpha}$. The maximum for the range $\delta^{1/2} \leq r \leq \delta^{\frac{\alpha}{2}-1}$ also occurs near $r= \delta^{1/2}$ (since the ball cannot intersect multiple ``tracks''), and is $\lesssim \delta^{\frac{3}{2}-\alpha}$. The maximum for the range $\delta^{\frac{\alpha}{2}-1} \leq r \leq 1$ occurs near $r=1$, and is $\lesssim \delta^{\frac{3}{2}-\alpha}$.  Applying the assumed inequality \eqref{traintrack} gives 
\[ \delta^{\frac{1}{2}-\alpha} \lesssim \delta^{3-2\alpha}. \]
Letting $\delta \to 0$ gives $\frac{1}{2}-\alpha \geq 3 - 2\alpha$ or $\alpha \geq 5/2$.    \end{proof}

It may be possible to weaken the requirement $\alpha \geq 5/2$ by replacing the $L^2$ norm on the left in the previous examples by an $L^p$ norm with $1 < p < 2$. Proposition~\ref{wavepacket2} below shows that at the critical exponent $\alpha = 2$, the average $L^p$ norms of $\delta$-discretised measures cannot be bounded by $C_{\epsilon} \delta^{-\epsilon}c_2(\mu)^{p}$ unless $p \leq 3/2$. The example is similar to the ``train tracks of planks'' used in Proposition~\ref{traintrackprop}, except that as $\alpha$ approaches 2, the spaces between the parallel ``train tracks'' tends to a distance $\sim 1$, and the example reduces to a single ``train track of planks'', which is a plank version of the original ``train track'' example from \cite[p.~563]{wolff} and \cite[p.~151]{katztao}.  When $\alpha = 2$, a uniform bound (without the $C_{\epsilon} \delta^{-\epsilon}$ factor) is known to not be possible; by considering a purely 2-unrectifiable set in $\mathbb{R}^3$ and applying the Besicovitch-Federer projection theorem. Finally, the counterexample for projections onto lines below is based on the Knapp example from \cite{erdogan} used to prove sharpness of the $L^2$ conical decay rates of 1-dimensional fractal measures.

\begin{proposition} \label{wavepacket2}  
\begin{enumerate} Let $p \in [1, \infty)$ and let $\epsilon >0$. 
\item If, for any $\delta>0$, 
\begin{equation} \label{hypothesis2} \int_0^{2\pi}  \left\lVert \pi_{\theta \#} \mu \right\rVert_{L^p(\gamma(\theta)^{\perp})}^p \, d\theta \leq C_{\epsilon}\delta^{-\epsilon} c_2(\mu)^{p}, \end{equation}
for all Borel measures $\mu$ on $B_3(0,1)$ of the form 
\[ \mu = \frac{1}{\delta^3\mathcal{H}^3(B(0,1))} \sum_{B \in \mathcal{B}} a_B \chi_B,  \]
where $a_B > 0$ for all $B \in \mathcal{B}$, and where $\mathcal{B}$ is a disjoint family of $\delta$-balls,  then $p \leq 3/2+\epsilon$.
\item If, for any $\delta>0$,
\begin{equation} \label{hypothesis27} \int_0^{2\pi}  \left\lVert \rho_{\theta \#} \mu \right\rVert_{L^p(\spn(\gamma(\theta)))}^p \, d\theta \leq C_{\epsilon}\delta^{-\epsilon} c_1(\mu)^{p}, \end{equation}
for all Borel measures $\mu$ on $B_3(0,1)$ of the form 
\[ \mu = \frac{1}{\delta^3\mathcal{H}^3(B(0,1))} \sum_{B \in \mathcal{B}} a_B \chi_B, \]
where $a_B > 0$ for all $B \in \mathcal{B}$, and where $\mathcal{B}$ is a disjoint family of $\delta$-balls,  then $p \leq 3/2+\epsilon$. \end{enumerate} \end{proposition}
\begin{proof} Let $\delta>0$ be small, fix $\theta_0 \in [0, 2\pi)$ and assume that \eqref{hypothesis2} holds. Let $\mu = \mu_{\delta, \theta_0}$ be the measure from the proof of Proposition~\ref{wavepacket}, given by 
\[ \mu(x) = \delta^{-3/2}\psi\left( \langle x, \gamma(\theta_0) \rangle,  \left\langle \delta^{-1/2}x, \sqrt{2} \gamma'(\theta_0) \right\rangle, \left\langle \delta^{-1} x, \sqrt{2} \left(\gamma \times \gamma'\right)(\theta_0) \right\rangle \right), \]
where $\psi$ is a smooth non-negative bump function supported in $B(0,1)$, with $\psi \sim 1$ on $B(0,1/2)$, such that $\int \psi = 1$. Similarly to the proof of Proposition~\ref{wavepacket},
\begin{align} \notag &\int_0^{2\pi} \left\lVert \pi_{\theta \#} \mu \right\rVert_{L^p(\gamma(\theta)^{\perp})}^p \, d\theta \\
\notag &\geq \int_{|\theta-\theta_0| \leq c\delta^{1/2}} \int_{-10^{-3}\delta^{1/2}}^{-10^{-3}\delta^{1/2}} \int_{-10^{-3}\delta}^{10^{-3}\delta} \delta^{-3p/2} \, dx_1 \, dx_2 \, d\theta \\
\label{pause38} &\gtrsim \delta^{2-\frac{3p}{2}}. \end{align}
But $\mu(\mathbb{R}^3) =1$, and
\[ c_2(\mu) \sim \delta^{-1/2}. \]
Moreover, $\mu$ is essentially a sum of indicator functions over a boundedly overlapping family of $\delta$-balls. Hence, if \eqref{hypothesis2} holds with the smaller $\mu(\mathbb{R}^3) c_2(\mu)^{p-1}$ on the right-hand side instead of $c_2(\mu)^{p}$, then
\[ \delta^{2-\frac{3p}{2}} \lesssim \delta^{\frac{1}{2}-\frac{p}{2}- \epsilon}. \]
Letting $\delta \to 0$ gives $2-\frac{3p}{2} \geq \frac{1}{2}-\frac{p}{2}-\epsilon$ or $p \leq 3/2+\epsilon$.  

In order to get the same restriction on $p$ if the larger right-hand side $c_2(\mu)^{p}$ is assumed, let $\nu$ be the sum of translated copies of $\mu$ spacing $\delta^{1/2}$ in the short direction $(\gamma \times \gamma')(\theta_0)$, given by 
\[ \nu(x) = \sum_{|m| \leq 10^{-3} \delta^{-1/2}} \mu\left( x- \delta^{1/2} m \left(\gamma \times \gamma'\right)(\theta_0)  \right). \]
There are $\sim \delta^{-1/2}$ such copies, and the supports of the projections of these translated copies under $\pi_{\theta}$ are pairwise disjoint for $|\theta-\theta_0| \leq c\delta^{1/2}$, for a sufficiently small absolute constant $c$, so the lower bound of $\delta^{2-\frac{3p}{2}}$ from \eqref{pause38} is multiplied by $\delta^{-1/2}$ to get 
\begin{equation} \label{etceq} \int_0^{2\pi} \left\lVert \pi_{\theta \#} \nu \right\rVert_{L^p(\gamma(\theta)^{\perp})}^p \, d\theta \gtrsim \delta^{\frac{3-3p}{2}}. \end{equation}
But $\nu(\mathbb{R}^3) \sim \delta^{-1/2} \sim c_2(\nu)$ (as shown in the proof of Proposition~\ref{traintrackprop}), and $\nu$ is essentially a sum of indicator functions over a boundedly overlapping family of $\delta$-balls.  Hence, by the assumed \eqref{hypothesis2},
\[ \delta^{\frac{3-3p}{2}} \lesssim \delta^{\frac{-p}{2}- \epsilon}. \]
Letting $\delta \to 0$ gives $\frac{3-3p}{2} \geq \frac{-p}{2}-\epsilon$ or $p \leq 3/2+\epsilon$.  This proves the restriction on $p$ for projections onto planes.

It remains to consider the restriction on $p$ for projections onto lines. Let $\mu = \mu_{\delta, \theta_0}$ be the measure from the proof of Proposition~\ref{wavepacket}, given by 
\[ \mu(x) = \delta^{-3/2}\psi\left( \langle x, \gamma(\theta_0) \rangle,  \left\langle \delta^{-1/2}x, \sqrt{2} \gamma'(\theta_0) \right\rangle, \left\langle \delta^{-1} x, \sqrt{2} \left(\gamma \times \gamma'\right)(\theta_0) \right\rangle \right), \]
where $\psi$ is a smooth non-negative bump function supported in $B(0,1)$, with $\psi \sim 1$ on $B(0,1/2)$, such that $\int \psi = 1$. Then $\mu(\mathbb{R}^3) \sim c_1(\mu) \sim 1$, and 
\[ \int_{[\theta_0-\delta^{1/2}, \theta_0 + \delta^{1/2} ] } \int |\rho_{\theta\#} \mu|^{p} \, d\theta \gtrsim \delta^{\frac{3}{2} - p}. \] 
Hence, by the assumed \eqref{hypothesis27},
\[ \delta^{\frac{3}{2} - p} \lesssim \delta^{-\epsilon}. \]
Letting $\delta \to 0$ gives $p \leq 3/2+\epsilon$. This proves the restriction on $p$ for projections onto lines.\end{proof}

For $(x,y,t ) \in \mathbb{R}^3$, define $\ell^*(x,y,t)$ to be the line
\[ \ell^*(x,y,t) = \left(0, x , t-\frac{xy}{2}\right) + L_y, \]
where $L_y$ is the ``light ray'' in the light cone 
\[ \widetilde{\Gamma} := \left\{ \eta \in \mathbb{R}^3 : \eta_2^2 = 2\eta_1\eta_3 \right\}, \]
given by
\[ L_y = \left\{ \lambda \left(1, -y, \frac{y^2}{2} \right) : \lambda \in \mathbb{R} \right\}. \]
The cone $\widetilde{\Gamma}$ is the image of the light cone 
\[ \Gamma := \left\{ \xi \in \mathbb{R}^3 : \xi_3^2 = \xi_1^2 + \xi_2^2 \right\}, \]
 in $\mathbb{R}^3$, under the orthogonal transformation $U(\xi_1, \xi_2,\xi_3)  = (\eta_1, \eta_2, \eta_3)$ given by 
\[ \eta_1 = \frac{\xi_1 + \xi_3}{\sqrt{2}}, \qquad \eta_3 = \frac{-\xi_1+\xi_3}{\sqrt{2}}, \qquad \eta_2 = \xi_2. \]
The restriction of $\ell^*$ to $|y| \leq \sqrt{2}$, followed by $U^*$ (an anti-clockwise rotation by $\pi/4$ in the $(\xi_1, \xi_3)$-plane), parametrises the family of light rays parallel to some $\gamma(\theta)$ with $|\theta| \leq \pi/2$. 

Given $(a,b,c) \in \mathbb{R}^3$, define $\ell(a,b,c)$ to be the horizontal line
\[ \ell(a,b,c) = \left\{ (b,0, c) + s\left(a, 1, b/2 \right) : s \in \mathbb{R} \right\}. \]
The following is the point-line duality principle from \cite{fasslerorponen}. 
\begin{lemma}[{\cite[Lemma~4.11]{fasslerorponen}}] \label{pointlineduality} Let $p \in \mathbb{R}^3$ and $p^* \in \mathbb{H}$. Then 
\[ p \in \ell^*(p^*) \quad \text{ if and only if } \quad p^* \in \ell(p). \] \end{lemma}
The symmetry of the lemma above means that much of the ideas from~\cite{fasslerorponen} for vertical projections in the Heisenberg group can be reversed as in the lemma below. 
\begin{lemma} \label{compmeasure} The measure $\mathfrak{m}$ given by 
\[ \mathfrak{m}(F) = \int_{-\pi/2}^{\pi/2} \mathcal{H}^2\left\{ y \in \gamma(\theta)^{\perp} : \pi_{\theta}^{-1}(y) \in U^*F \right\} \, d\theta, \]
for a Borel set $F$ of light rays parallel to lines in $\widetilde{\Gamma}$, is comparable to the pushforward of the Lebesgue measure on $\mathbb{R}^3 \cap \left\{|y| \leq \sqrt{2}\right\}$ under the map $\ell^*$, meaning that
\begin{equation} \label{comparablemeasure} \left(\ell^*_{\#}\mathcal{H}^3\chi_{|y| \leq \sqrt{2}}\right)(F) \sim \mathfrak{m}(F), \end{equation}
for any Borel set $F$ of light rays parallel to lines in $\widetilde{\Gamma}$. As a consequence, for any non-negative Borel function $f$ on the set of light rays parallel to lines in $\widetilde{\Gamma}$,
\begin{equation} \label{functionversion} \int_{-\pi/2}^{\pi/2} \int_{\gamma(\theta)^{\perp}} f(U \pi_{\theta}^{-1}(y) ) \, d\mathcal{H}^2(y) \, d\theta \sim \int f \, d\left(\ell^*_{\#}\mathcal{H}^3\chi_{|y| \leq \sqrt{2}}  \right). \end{equation} \end{lemma}
\begin{proof} Let $F$ be given. Let $m$ be the Lebesgue measure on $\mathbb{R}^3 \cap \left\{|y| \leq \sqrt{2}\right\}$. Then 
\begin{align*} (\ell^*_{\#}m)(F) &= m\left\{ (x,y,t) \in \mathbb{R}^3 : |y| \leq \sqrt{2},  \left(0, x , t-\frac{xy}{2}\right) + L_y \in F\right\} \\
&= \int_{-\sqrt{2}}^{\sqrt{2}} \mathcal{H}^2\left\{ (x,t) \in \mathbb{R}^2 : \left(0, x , t-\frac{xy}{2}\right) + \spn(1,-y,y^2/2) \in F \right\} \, dy \\
&\sim \int_{-\sqrt{2}}^{\sqrt{2}} \mathcal{H}^2\left\{ x \in  (1,-y,y^2/2)^{\perp} : \pi_{(1,-y,y^2/2)^{\perp}}^{-1}(x) \in F \right\} \, dy \\
&= \int_{-\sqrt{2}}^{\sqrt{2}} \mathcal{H}^2\left\{ x \in (U^*(1,-y,y^2/2))^{\perp} : \pi_{(U^*(1,-y,y^2/2))^{\perp}}^{-1}(x) \in U^*F \right\} \, dy \\
&\sim \int_{-\pi/2}^{\pi/2} \mathcal{H}^2\left\{ x \in (\cos \theta, \sin \theta, 1)^{\perp} : \pi_{(\cos \theta, \sin \theta, 1)^{\perp}}^{-1}(x) \in U^*F \right\} \, d\theta \\
&= \mathfrak{m}(F). \end{align*}
This proves \eqref{comparablemeasure}, and yields \eqref{functionversion} whenever $f = \chi_F$ for a Borel set $F$ of light rays. The equivalence \eqref{functionversion} for general Borel functions follows by approximating $f$ with simple functions and applying the monotone convergence theorem.    \end{proof}
\begin{proposition}  If, for any $\epsilon >0$, 
\begin{equation} \label{kakeya} \left\lVert  \sum_{T \in \mathbb{T} } \chi_T \right\rVert_{L^p(B(0,1))} ^p \leq C_{\epsilon} \delta^{-\epsilon}, \end{equation}
for any family $\mathbb{T}$ of $1 \times \delta \times \delta$ $SL_2$ tubes with $\delta$-separated directions, then $p \leq 3/2$. 
\end{proposition}
\begin{proof} It is first shown that if, for any $\epsilon >0$, \eqref{kakeya} holds for any family $\mathbb{T}$ of $1 \times \delta \times \delta$ $SL_2$ tubes satisfying the \emph{weaker} property that \[ |\{ T \in \mathbb{T} : T \subseteq S \} | \leq (r/ \delta)^2, \]
for any tube $S$ of length 2 and radius $r$, and any $r \in [\delta, 1]$, then $p \leq 3/2$. 

Let $\delta>0$ be small and choose $\theta_0 = \pi/2$. Let $\mu = \mu_{\delta, \theta_0}$ be a $\delta^{1/2}$-multiple of the (slightly modified) measure from the proof of Proposition~\ref{wavepacket}, given by 
\[ \mu(x) = \delta^{-1}\psi\left( \langle x, \gamma(\theta_0) \rangle,  \left\langle \delta^{-1/2}x, \sqrt{2} \gamma'(\theta_0) \right\rangle, \left\langle \delta^{-1} x, \sqrt{2} \left(\gamma \times \gamma'\right)(\theta_0) \right\rangle \right), \]
where $\psi$ is a smooth non-negative bump function supported in $B(0,c)$, with $\psi \sim 1$ on $B(0,c/2)$, such that $\int \psi = 1$ (here $c$ is a small absolute constant to be chosen; this is the only difference from the proof of Proposition~\ref{wavepacket} where $c=1$). As in the proof of Proposition~\ref{wavepacket2}, let $\nu$ be the sum of translated copies of $\mu$ spacing $\delta^{1/2}$ in the short direction $(\gamma \times \gamma')(\theta_0)$, given by 
\[ \nu(x) = \sum_{|m| \leq c \delta^{-1/2}} \mu\left( x- \delta^{1/2} m \left(\gamma \times \gamma'\right)(\theta_0)  \right). \]
Then $\nu$ is supported in a ball around the origin of radius $\sim c$, and $\nu$ is very similar to a $\delta^{1/2}$-multiple of the measure from the proof of Proposition~\ref{wavepacket2}, so by a similar argument to \eqref{etceq},
\[ \delta^{\frac{3}{2}-p} \lesssim \int_{-\pi/2}^{\pi/2}  \left\lVert \pi_{\theta \#} \nu \right\rVert_{L^p(\gamma(\theta)^{\perp})}^p \, d\theta. \]
 Let $\mathcal{B}$ be a family of $\delta$-balls, such that the centres of the balls in $\mathcal{B}$ form a maximal $\delta$-separated subset of the support of $\nu$. Then 
\[ \nu \leq  \sum_{B \in \mathcal{B}} \delta^{-1} \chi_B. \]
Hence, by Lemma~\ref{compmeasure},
\begin{align}\notag \int_{-\pi/2}^{\pi/2}  \left\lVert \pi_{\theta \#} \nu \right\rVert_{L^p(\gamma(\theta)^{\perp})}^p \, d\theta  &= \int_{-\pi/2}^{\pi/2} \int_{\gamma(\theta)^{\perp}} \left\lvert \int_{\pi_{\theta}^{-1}(y) } \nu \, d\mathcal{H}^1 \right\rvert^p \, d\mathcal{H}^2(y) \, d\theta\\
\notag &= \int_{-\pi/2}^{\pi/2} \int_{\gamma(\theta)^{\perp}} \left\lvert \int_{U\pi_{\theta}^{-1}(y) } (U_{\#}\nu) \, d\mathcal{H}^1 \right\rvert^p \, d\mathcal{H}^2(y) \, d\theta\\
\notag &\sim \int \left\lvert \int_L U_{\#}\nu \, d\mathcal{H}^1 \right\rvert^p \,  d(\ell^*_{\#}\mathcal{H}^3)(L) \\
\label{presimplify} &\lesssim \int \left\lvert \left\{ B \in \mathcal{B} : UB \cap L \neq \emptyset \right\} \right\rvert^p \,  d(\ell^*_{\#}\mathcal{H}^3)(L) \\
\label{simplify} &= \int_{|p_2^*| \leq \sqrt{2}} \left\lvert \left\{ B \in \mathcal{B} : UB \cap \ell^*(p^*) \neq \emptyset \right\} \right\rvert^p \,  d\mathcal{H}^3(p^*). \end{align}
If $|p_2^*| \leq \sqrt{2}$ and $\ell^*(p^*)$ intersects $B(0,c)$, then $|p^*| \leq 2$ (provided $c$ is now chosen sufficiently small; this follows easily from the definition of $\ell^*$). Hence 
\begin{equation} \label{simplified} \eqref{simplify} \leq  \int_{|p^*| \leq 2} \left\lvert \left\{ B \in \mathcal{B} : UB \cap \ell^*(p^*) \neq \emptyset \right\} \right\rvert^p \,  d\mathcal{H}^3(p^*). \end{equation}
By Lemma~\ref{pointlineduality} (point-line duality), 
\[ \eqref{simplified} \leq \int_{|p^*| \leq 2} \left\lvert \sum_{B \in \mathcal{B}} \chi_{\ell(UB)}(p^*) \right\rvert^p \,  d\mathcal{H}^3(p^*). \]
The measure $\nu$ satisfies $c_2(\nu) \lesssim 1$, and similarly so does the measure $\sum_{B \in \mathcal{B}} \delta^{-1} \chi_B$ (which is roughly the same as $\nu$), which means that for any $\delta \leq r \leq 1$, the number of $\delta$-balls from $\mathcal{B}$ intersecting any $r$-ball is $\lesssim (r/\delta)^2$. By the formula for $\ell$, this implies that for any $r \in [\delta, 1]$ and any tube $T$ of dimensions $2 \times r \times r$, 
\[ \left\lvert\left\{ \ell(UB) : B \in \mathcal{B}, \ell(UB) \subseteq T \right\} \right\rvert \lesssim (r/\delta)^2. \]
Therefore, if \eqref{kakeya} holds for any family of $SL_2$ tubes satisfying a 2-dimensional ball condition, then
\[ \delta^{\frac{3}{2}-p} \lesssim \int_{|p^*| \leq 2} \left\lvert \sum_{B \in \mathcal{B}} \chi_{\ell(UB)}(p^*) \right\rvert^p \,  d\mathcal{H}^3(p^*) \leq C_{\epsilon} \delta^{-\epsilon}. \]
This yields $p \leq 3/2 + \epsilon$. 

To get the same restriction on $p$ when \eqref{kakeya} holds only for those sets $\mathbb{T}$ of $1 \times \delta \times \delta$ $SL_2$ tubes with $\delta$-separated directions, it suffices to show that the collection of tubes $\ell(UB)$ have $\sim \delta$-separated directions, for any balls separated by $100\delta$. Let $B_1, B_2 \in \mathcal{B}$ be $100\delta$-separated balls in $\mathcal{B}$. Let $(a_j,b_j,c_j)$ be the centre of $UB_j$, where $j \in \{1,2\}$. The direction of $\ell(UB_j)$ is parallel to $(a_j, 1, b_j/2)$, so it is required to show that 
\[ |a_1-a_2| + |b_1-b_2| \gtrsim \delta. \]
Write
\[ B_2 = B_1 + \lambda_1 \gamma(\pi/2) + \lambda_2 \gamma'(\pi/2) + \lambda_3 \left( \gamma \times \gamma'\right)(\pi/2), \]
where $|\lambda_3| \leq 1$, $|\lambda_2| \leq c \delta^{1/2}$ and $|\lambda_1| \leq 1$, and $|(\lambda_1, \lambda_2, \lambda_3)| \geq 100\delta$. This can be written as 
\[ B_2 = B_1 + \frac{\lambda_1}{\sqrt{2}}( 0, 1, 1) + \frac{\lambda_2}{\sqrt{2}} (-1, 0, 0) + \frac{\lambda_3}{2} (0, -1, 1 ). \]
Hence 
\[  UB_2 = UB_1 + \frac{\lambda_1}{\sqrt{2}}\left( \frac{1}{\sqrt{2}}, 1, \frac{1}{\sqrt{2}}\right) + \frac{\lambda_2}{\sqrt{2}}\left( \frac{-1}{\sqrt{2}}, 0, \frac{1}{\sqrt{2}}\right) + \frac{\lambda_3}{2}\left( \frac{1}{\sqrt{2}}, -1, \frac{1}{\sqrt{2}}\right). \]
It follows that 
\[ (a_2,b_2) = (a_1,b_1) + \left( \frac{\lambda_1}{2} - \frac{\lambda_2}{2} + \frac{\lambda_3}{2\sqrt{2} }, \frac{\lambda_1}{\sqrt{2}} - \frac{\lambda_3}{2} \right). \]
By the triangle inequality,
\begin{align*}  \left\lvert \frac{\lambda_1}{2} - \frac{\lambda_2}{2} + \frac{\lambda_3}{2\sqrt{2} } \right\rvert &\geq \left\lvert \frac{\lambda_3}{\sqrt{2}}  - \frac{\lambda_2}{2} \right\rvert - \frac{1}{\sqrt{2}} \left\lvert \frac{\lambda_1}{\sqrt{2}}  - \frac{\lambda_3}{2}\right\rvert. \end{align*}
Hence 
\[ |a_1-a_2| + \frac{|b_1-b_2|}{\sqrt{2}} \geq \left\lvert \frac{\lambda_3}{\sqrt{2}} - \frac{\lambda_2}{2} \right\rvert. \]
If $B_1$ and $B_2$ lie in different planks, then $|\lambda_3| \geq \delta^{1/2}$, and since $|\lambda_2| \leq c \delta^{1/2}$, this yields 
\[ |a_1-a_2| + |b_1-b_2| \gtrsim \delta^{1/2}. \]
It remains to consider the case where $B_1$ and $B_2$ lie in the same plank. In this case, $|\lambda_3| \leq \delta$. If $|\lambda_1| \geq 10\delta$, then $|b_1-b_2| \sim \left\lvert \frac{\lambda_1}{\sqrt{2}}-\frac{\lambda_3}{2}\right\rvert \gtrsim \delta$. Otherwise $|\lambda_1| < 10\delta$ and $|\lambda_3| \leq \delta$ imply that $|\lambda_2| \geq 50\delta$, and this gives $|a_1-a_2| = \left\lvert \frac{\lambda_1}{2} - \frac{\lambda_2}{2} + \frac{\lambda_3}{2\sqrt{2} } \right\rvert \gtrsim \delta$. This shows that the directions of $100\delta$-separated balls in $\mathcal{B}$ are $\sim \delta$-separated. It follows that if \eqref{kakeya} holds for any family $\mathbb{T}$ of $1 \times \delta \times \delta$ $SL_2$ tubes with $\delta$-separated directions, then it can be applied to the family $\{ \ell(UB) : B \in \mathcal{B} \}$, and therefore
\[ \delta^{\frac{3}{2}-p} \lesssim \int_{|p^*| \leq 2} \left\lvert \sum_{B \in \mathcal{B}} \chi_{\ell(UB)}(p^*) \right\rvert^p \,  d\mathcal{H}^3(p^*) \leq C_{\epsilon} \delta^{-\epsilon}. \]
This yields $p \leq 3/2 + \epsilon$. Letting $\epsilon \to 0$ gives $p \leq 3/2$. 
  \end{proof}
	
	\section{An \texorpdfstring{$L^{4/3}$}{L4/3} inequality for projections onto planes} \label{Lpbounds}
	
Consider the light cone $\{\xi \in \mathbb{R}^3 : |(\xi_1,\xi_2)| = |\xi_3| \}$. The set $\{ \xi \in \mathbb{R}^3 : 2^j \gtrsim |(\xi_1,\xi_2)| \geq |\xi_3| \sim 2^j \}$ will be broken into dyadic conical shells according to the distance to the cone. If this set is rescaled by $2^{-j}$, then this is equivalent to dividing the $\sim 1$ neighbourhood of the truncated light cone $|\xi_3| = |(\xi_1, \xi_2) |$ in $B(0,100) \setminus B(0,1/100)$ into dyadic conical shells according to the distance to the cone. The purpose of this decomposition is that on each dyadic conical shell, the change of variables used in what follows will have constant Jacobian on that shell.  The part with distance $\sim 2^{j-k}$ from the cone (before scaling), where $0 \leq k < j$, is roughly a scaling of the $\sim 2^{-k}$-neighbourhood of the truncated light cone (though the outer part, separated by a distance $\sim 2^{-k}$ from the cone) by $2^j$. The scaled down neighbourhood can be covered by a finitely overlapping collection of boxes similar to the cover used for the $2^{-k}$-neighbourhood of the truncated light cone; of dimensions $\sim 1 \times 2^{-k/2} \times 2^{-k}$, and this covering can then be scaled up. If $k=j$ then the shell with distance $\lesssim 1$ from the cone (inside $B(0, 100 \cdot 2^j ) \setminus B(0, 2^j/100)$) is  a rescaling of the $\sim 2^{-j}$-neighbourhood of the truncated light cone. This neighbourhood can be covered by the standard $\lesssim 1$ overlapping cover of the $2^{-j}$-neighbourhood of the truncated light cone by boxes of dimensions $1 \times 2^{-j/2} \times 2^{-j}$, and this covering can then be scaled up. 

The decomposition will be made a bit more precise. For each $\theta \in [0, 2\pi)$ and $0 \leq k < j$, let 
\begin{multline*} \tau(\theta,j,k) =\\
 \{ \lambda_1 (\gamma \times \gamma')(\theta) + \lambda_2 \gamma'(\theta) + \lambda_3 \gamma(\theta) : |\lambda_1| \sim 2^j, |\lambda_2| \lesssim 2^{j - k/2}, |\lambda_3| \sim 2^{j-k} \}, \end{multline*}
and for $k=j$ let 
\[ \tau(\theta,j,j) = \{ \lambda_1 (\gamma \times \gamma')(\theta) + \lambda_2 \gamma'(\theta) + \lambda_3 \gamma(\theta) : |\lambda_1| \sim 2^j, |\lambda_2| \lesssim 2^{j/2}, |\lambda_3| \lesssim 1 \}. \]
For each $k$, choose a maximal $\sim 2^{-k/2}$-separated set $\Theta_k$ of $[0, 2\pi)$, such that the sets $\Lambda_{j,k} = \{ \tau(\theta,j, k) : \theta \in \Theta_k \}$ form a $\lesssim 1$ overlapping cover of the $\sim 2^{j-k}$ outer neighbourhood of the light cone $\{ (\xi, |\xi|) : |\xi| \sim 2^j \}$. The set $ \Lambda = \bigcup_{j \geq 1} \bigcup_{0 \leq k \leq j} \Lambda_{j,k}$ forms a boundedly overlapping cover of the set $\{ \xi \in \mathbb{R}^3 : 1 \lesssim |\xi_3 | \leq |(\xi_1, \xi_2)| \lesssim |\xi| \}$. Let $\{ \psi_{\tau}\}_{\tau \in \Lambda }$ be a smooth partition of unity such that each $\psi_{\tau}$ is supported in $\tau$ and such that the $\psi_{\tau}$'s sum to 1 on the set $\{ \xi \in \mathbb{R}^3 : 1 \lesssim |\xi_3 | \leq |(\xi_1, \xi_2)| \lesssim |\xi| \}$. It may also be assumed that the derivatives of $\psi_{\tau}$ are of the expected size (depending on the dimensions of $\tau$). Given $\delta>0$ (which may be thought of as vanishingly small), for each $\tau$ cover $B(0,1)$ by a boundedly overlapping collection of (rescaled) planks $T \in \mathbb{T}_{\tau}$ of dimensions $2^{k\delta -(j-k)} \times 2^{k (-1/2+\delta) -(j-k)} \times 2^{k(-1 + \delta) -(j-k)}$ dual to $\tau$ (meaning that the long direction of $T$ is the short direction of $\tau$, and vice-versa). Given a measure $\mu$ on $B(0,1)$ and a plank $T$, define 
\[ M_T \mu = \eta_T \left( \mu \ast \widecheck{\psi_{\tau} } \right). \]

Lemma~\ref{energylemma} below states that the projections of measures of dimension greater than 2 have small $L^{4/3}$ norms on small unions of rectangles, on average. The main use of Lemma~\ref{energylemma} will be in proving that pushforwards of measures of dimension greater than 2 are almost surely in $L^{4/3}$; by splitting $\mu$ into ``good'' and ``bad'' parts, and bounding the $L^{4/3}$ norm of the projections of the ``bad'' part of a measure by the $L^{4/3}$ norm of the projections of the original measure on small unions of rectangles. 	

Let $\phi$ be a smooth bump function equal to 1 on $B_3(0,1)$ and vanishing outside $B_3(0,2)$. For each $R >0$ let $\phi_R(x) = R^3\phi(Rx)$. 
	
	\begin{lemma}  \label{energylemma} Let $\alpha > 2$ and let $p=4/3$. For any $\beta >0$, there exists an $\varepsilon>0$ (depending on $\alpha$ and $\beta$) such that 
\begin{equation} \label{geometric} \int_0^{2\pi} \int_{\bigcup_{D \in \mathbb{D}_{\theta}}} \left\lvert \pi_{\theta \#}(\mu \ast \phi_R)\right\rvert^p \, d\mathcal{H}^2 \, d\theta  \leq C_{\alpha,\beta} R^{-\varepsilon} \mu(\mathbb{R}^3) c_{\alpha}(\mu)^{p-1}, \end{equation}
for any $R > 1$, for any Borel measure $\mu$ on $B_3(0,1)$ with $c_{\alpha}(\mu) < \infty$, and for any family of sets $\{ \mathbb{D}_{\theta} \}$, where each $\mathbb{D}_{\theta}$ is a set of rectangles of dimensions $R^{-1/2} \times R^{-1}$ in $\gamma(\theta)^{\perp}$ with long direction parallel to $\gamma'(\theta)$ and short direction parallel to $(\gamma \times \gamma')(\theta)$, such that each $\mathbb{D}_{\theta}$ has cardinality $\lvert\mathbb{D}_{\theta}\rvert \leq R^{3/2-\beta} \mu(\mathbb{R}^3)c_{\alpha}(\mu)^{-1}$, and such that the integrand in \eqref{geometric} is Borel measurable. \end{lemma} 
\begin{remark} Having $\mu \ast \phi_R$ instead of $\mu$ in \eqref{geometric} is equivalent to requiring that $\mu$ is constant at scale $R^{-1}$.  \end{remark}
\begin{proof} The proof is a bootstrapping argument which gradually lowers the value of $\beta$, which for simplicity will be written as a proof by contradiction. Let $\beta_{\inf}$ be the infimum over all positive $\beta$ for which the conclusion of the lemma is true. Suppose for a contradiction that $\beta_{\inf} >0$. Let $\beta_0 \in (\beta_{\inf}, 2 \beta_{\inf} )$. Clearly $\beta_{\inf} \leq 3/2$ (in particular it is finite); this follows from the inequality $\mu(\mathbb{R}^3) \leq c_{\alpha}(\mu)$. To obtain a contradiction, it will suffice to show that the lemma holds for all $\beta > c\beta_0$, where $c\in (0,1)$ is chosen sufficiently close to 1 to ensure that $1-c < \frac{\alpha-2}{100\beta_{\inf}}$. Let $\beta > c \beta_0$. Let $\delta>0$ be such that $\delta \ll \varepsilon(\beta_0)$ and $\delta \ll \beta$. The measure $\mu \ast \phi_R$ will be re-labelled as $\mu$ to simplify the notation (so it may be assumed that $\widehat{\mu}$ is negligible outside $B(0, R^{1+\delta})$). Let $J$ be such that $2^J \sim R^{\epsilon}$, where $\epsilon = 10^{-10} \min\{\beta, \alpha-2\}$.  Given $(j,k)$ and $\tau \in \Lambda_{j,k}$, let 
\begin{equation} \label{badplanks} \mathbb{T}_{\tau,b} = \left\{ T \in \mathbb{T}_{\tau} : \mu(4T) \geq 2^{-k \left(\frac{3}{2} -\beta_0\right)} 2^{-(j-k)\alpha - 10^{10} k \delta }  c_{\alpha}(\mu)\right\}, \end{equation}
and let
\begin{equation} \label{badpart} \mu_b = \sum_{j \geq J} \sum_{k \in [ j\epsilon, j] } \sum_{\tau \in \Lambda_{j,k} } \sum_{T \in \mathbb{T}_{\tau,b} }  M_T \mu, \qquad \mu_g = \mu - \mu_b. \end{equation}
 By the triangle inequality, followed by Hölder's inequality,
\begin{multline} \label{asterisk} \int_0^{2\pi} \int_{\bigcup_{D \in \mathbb{D}_{\theta}}} \left\lvert \pi_{\theta \#}\mu\right\rvert^p \, d\mathcal{H}^2 \, d\theta \lesssim \int_0^{2\pi} \int \left\lvert \pi_{\theta \#}\mu_b\right\rvert^p \, d\mathcal{H}^2 \, d\theta \\
+ \left(\frac{ \mu(\mathbb{R}^3 )}{c_{\alpha}(\mu)} \right)^{1- \frac{p}{2} }R^{-\beta\left( 1- \frac{p}{2} \right) } \left(\int_0^{2\pi} \int \left\lvert \pi_{\theta \#}\mu_g\right\rvert^2 \, d\mathcal{H}^2 \, d\theta\right)^{p/2}. \end{multline}
In what follows, negligible tail terms of the form $R^{-N} \mu(\mathbb{R}^3)$ will be left out, since whenever such terms dominate, the desired inequalities follow trivially. By the triangle inequality, followed by Hölder's inequality again,
\begin{multline*} \int_0^{2\pi} \int \left\lvert \pi_{\theta \#}\mu_b\right\rvert^p \, d\mathcal{H}^2 \, d\theta \\
\lesssim (\log R)^{O(1)} \sum_{j \geq J} \sum_{k \in [j\epsilon, j ] }   \int_0^{2\pi} \int \left\lvert \sum_{\tau \in \Lambda_{j,k}} \sum_{T \in \mathbb{T}_{\tau,b} }  \pi_{\theta \#} M_T \mu \right\rvert^p \, d\mathcal{H}^2 \, d\theta. \end{multline*}
Fix a pair $(j,k)$ occurring in the above sum. It will be shown that
\begin{multline} \label{decoupling}  \int_0^{2\pi}  \int \left\lvert \sum_{\tau \in \Lambda_{j,k}} \sum_{T \in \mathbb{T}_{\tau,b} }  \pi_{\theta \#} M_T \mu \right\rvert^p \, d\mathcal{H}^2 \,d \theta  \\
\lesssim 2^{kO(\delta)}  \int_0^{2\pi} \sum_{\substack{\tau \in \Lambda_{j,k}: \\ \left\lvert \theta_{\tau} - \theta \right\rvert \leq 2^{k\left(\delta-1/2\right)}}} \sum_{T \in \mathbb{T}_{\tau,b} }  \mu(2T)^p (m(T)2^{j-k})^{-(p-1)} \, d\theta, \end{multline}
where $\theta_{\tau}$ is defined by $\tau = \tau\left(\theta_{\tau}, j, k\right)$, and $m(T)$ is the Lebesgue measure of $T$. Assume first that $k=j$. The left-hand side of \eqref{decoupling} is 
\begin{multline*} \lesssim    \int_0^{2\pi} \int \left\lvert \sum_{\substack{\tau \in \Lambda_{j,k}: \\ \left\lvert \theta_{\tau} - \theta \right\rvert \leq 2^{k\left(\delta-1/2\right)}}} \sum_{T \in \mathbb{T}_{\tau,b} }  \pi_{\theta \#} M_T \mu \right\rvert^p \, d\mathcal{H}^2 \,d\theta \\
+
   \int_0^{2\pi} \int \left\lvert \sum_{\substack{\tau \in \Lambda_{j,k}: \\ \left\lvert \theta_{\tau} - \theta \right\rvert > 2^{k\left(\delta-1/2\right)}}} \sum_{T \in \mathbb{T}_{\tau,b} }  \pi_{\theta \#} M_T \mu \right\rvert^p \, d\mathcal{H}^2 \, d\theta. \end{multline*}
The second term is negligible by the trivial $L^{\infty}$ bound on each $\pi_{\theta \#} M_T \mu$, followed by the same stationary phase bound used in the $L^1$ case (see Lemma~5 of \cite{GGGHMW}). By the $\lesssim 2^{kO(\delta)}$-overlapping property of the $\pi_{\theta}(T)$'s, 
\begin{multline*}   \int_0^{2\pi}  \int \left\lvert \sum_{\substack{\tau \in \Lambda_{j,k}: \\ \left\lvert \theta_{\tau} - \theta \right\rvert \leq 2^{k\left(\delta-1/2\right)}}} \sum_{T \in \mathbb{T}_{\tau,b} }  \pi_{\theta \#} M_T \mu \right\rvert^p \, d\mathcal{H}^2 \,d \theta  \\
\lesssim 2^{kO(\delta)}  \int_0^{2\pi} \sum_{\tau \in \Lambda_{j,k}} \sum_{T \in \mathbb{T}_{\tau,b} }  \| \pi_{\theta \#} M_T \mu\|_{L^p(\mathcal{H}^2)}^p \, d\theta. \end{multline*}
By the uncertainty principle, the right-hand side of the above is bounded by the right-hand side of \eqref{decoupling}.

Now consider the case $k < j$. If $p$ were equal to 2 then the inequality
\begin{multline} \label{placeholder}   \int_0^{2\pi}  \int \left\lvert \sum_{\tau \in \Lambda_{j,k}} \sum_{T \in \mathbb{T}_{\tau,b} }  \pi_{\theta \#} M_T \mu \right\rvert^p \, d\mathcal{H}^2 \,d \theta  \\
\lesssim 2^{k(\delta)}  \int_0^{2\pi} \sum_{\tau \in \Lambda_{j,k}} \sum_{T \in \mathbb{T}_{\tau,b} }  \| \pi_{\theta \#} M_T \mu\|_{L^p(\mathcal{H}^2)}^p \, d\theta. \end{multline}
would follow by applying Plancherel in 2 dimensions, changing variables to an integral over $\mathbb{R}^3$, applying Plancherel in $\mathbb{R}^3$ to decouple the $L^2$ norms (as in \cite[Eq.~133]{GGGHMW}), reversing the change of variables, and then finally reversing the 2-dimensional Plancherel step. Fixing $(j,k)$ ensures that the Jacobian of the change of variables is the constant $2^{-j+k/2}$. If $p$ were equal to 1 then \eqref{placeholder} would follow from the triangle inequality (thus the case $p=2$ makes use of the averaging in $\theta$, but the case $p=1$ does not). This argument implies \eqref{decoupling} (more generally for any subcollection of $T$'s in the left-hand side) if either $p=1$ or $p=2$ (the right-hand side of \eqref{placeholder} is bounded by the right-hand side of \eqref{decoupling}; this follows straightforwardly from the uncertainty principle and stationary phase (Lemma~5 from \cite{GGGHMW})). The case $p = 4/3$ for \eqref{decoupling} holds by an application of the interpolation argument from \cite[Exercise~9.21]{demeter}; by pigeonholing it may be assumed that the sum over the $T$'s in the left-hand side of \eqref{decoupling} is over a subset of the $T$'s such that $\mu(2T)$ is constant up to a factor of 2, and then \eqref{decoupling} follows writing $p = (1-\phi) \cdot 1 + \phi \cdot 2$ and applying Hölder's inequality with $q = 1/(1-\phi)$ and $q' = 1/\phi$. This proves \eqref{decoupling}.

By Hölder's inequality applied to the right-hand side of \eqref{decoupling},
\begin{multline*} \int \left\lvert \pi_{\theta \#}\mu_b\right\rvert^p \, d\mathcal{H}^2 \lesssim (\log R)^{O(1)} \sum_{j \geq J} \sum_{k \in [j\epsilon, j ] } 2^{kO(\delta)} \sum_{m} \\
\int_0^{2\pi} \sum_{\substack{\tau \in \Lambda_{j,k}: \\ \left\lvert \theta_{\tau} - \theta \right\rvert \leq 2^{k\left(\delta-1/2\right)}}} \sum_{\substack{ T \in \mathbb{T}_{\tau,b} : \\
T \cap B_m \neq \emptyset } }  \int_{\pi_{\theta}(2T) } | \pi_{\theta\#}\mu_{m,j,k}|^p \, d\theta, \end{multline*}
where, for each pair $(j,k)$, $\{B_m\}_m$ be a boundedly overlapping cover of the unit ball by balls of radius $2^{-(j-k)+k\delta}$, and $\mu_{m,j,k}$ is the restriction of $\mu$ to $100 B_m$. Due the right-hand side of \eqref{decoupling}, the measure $\mu_{m,j,k}$ can be replaced by $\mu_{m,j,k} \ast \phi_{2^{j-k\delta}}$, which will be supressed in the notation. By rescaling each ball $B_m$ to a ball of radius $1$,
\begin{multline*} \int \left\lvert \pi_{\theta \#}\mu_b\right\rvert^p \, d\mathcal{H}^2 \lesssim (\log R)^{O(1)} \sum_{j \geq J} \sum_{k \in [j\epsilon, j ] } 2^{kO(\delta)} \sum_{m} \\
\int_0^{2\pi} \sum_{\substack{\tau \in \Lambda_{j,k}: \\ \left\lvert \theta_{\tau} - \theta \right\rvert \leq 2^{k\left(\delta-1/2\right)}}} \sum_{\substack{ T \in \mathbb{T}_{\tau,b} : \\
T \cap B_m \neq \emptyset } }  2^{2(j-k)(p-1)} \int_{\pi_{\theta}(2^{j-k-k\delta}2T) } \left\lvert \pi_{\theta\#} 2^{j-k-k\delta}_{\#} \mu_{m,j,k}\right\rvert^p \, d\theta, \end{multline*}
Each rescaled measure $2^{j-k-k\delta}_{\#} \mu_{m,j,k}$ and collection of rescaled planks $2^{j-k-k\delta}2T$ satisfies the hypotheses of the lemma; by the definition of $\beta_0$ and of the ``bad'' planks (and since $2^{j-k-k\delta}_{\#} \mu_{m,j,k}$ is essentially constant at scale $2^{-k}$). Moreover, $c_{\alpha}\left(2^{j-k-k\delta}_{\#} \mu_{m,j,k} \right) \leq 2^{-\alpha(j-k)+kO(\delta)} c_{\alpha}(\mu)$. Hence, by the conclusion of the lemma for $\beta_0$ and the assumption that $\alpha >2$, 
\begin{multline*} \int \left\lvert \pi_{\theta \#}\mu_b\right\rvert^p \, d\mathcal{H}^2 \\
 \lesssim (\log R)^{O(1)} \sum_{j \geq J} \sum_{k \in [j\epsilon, j ] } \sum_{m} \mu(100 B_m) 2^{kO(\delta)} c_{\alpha}(\mu)^{p-1} 2^{-k \varepsilon} \\
\lesssim 2^{-J \varepsilon} \mu(\mathbb{R}^3) c_{\alpha}(\mu)^{p-1} \lesssim R^{-\beta \varepsilon/10^4} \mu(\mathbb{R}^3) c_{\alpha}(\mu)^{p-1}.  \end{multline*}

It remains to bound the ``good'' part. By \eqref{asterisk}, it suffices to show that  
\begin{equation} \label{goodsufficient} \int_0^{2\pi} \int \left\lvert \pi_{\theta \#}\mu_g\right\rvert^2 \, d\mathcal{H}^2 \, d\theta \ll   \mu(\mathbb{R}^3) c_{\alpha}(\mu) R^{\beta \left(\frac{2}{p}-1 \right) }. \end{equation}
By the Fourier formula for the energy, the part away from the cone is bounded by $I_{2+100\epsilon}(\mu) \lesssim \mu(\mathbb{R}^3) c_{\alpha}(\mu)$ (this is the same as the bound for Eq.~107 from \cite{GGGHMW}). Therefore, by the same argument as in \cite{GGGHMW} (around \cite[Eq.113]{GGGHMW}), it suffices to prove that 
\[ \sum_{j \geq J} \sum_{k \in [\epsilon j, j ] } \sum_{\tau \in \Lambda_{j,k} } \sum_{T \in \mathbb{T}_{\tau, g} } 2^{-j + k/2+kO(\delta)} \|M_T \mu \|_2^2 \ll \mu(\mathbb{R}^3) c_{\alpha}(\mu) R^{\beta \left(\frac{2}{p}-1 \right) }. \]
The same argument as that in \cite{GGGHMW} gives (the $\alpha_0$ in \cite[Eq.~137]{GGGHMW} satisfies $\frac{\alpha_0+1}{2} = \frac{3}{2}-\beta_0$)
\begin{multline*} \sum_{j \geq J} \sum_{k \in [\epsilon j, j ] } \sum_{\tau \in \Lambda_{j,k} } \sum_{T \in \mathbb{T}_{\tau, g} } 2^{-j + k/2+kO(\delta)} \|M_T \mu \|_2^2 \lesssim  \mu(\mathbb{R}^3) c_{\alpha}(\mu) \times \\
\sum_{J \leq j \leq (1+\delta) \log_2(R) } \sum_{k \in [\epsilon j, j ] } 2^{-j + k/2+kO(\delta)} 2^{j (3-\alpha) + k \left( \frac{1}{2} - \frac{1}{p_{\dec}} \right) \left( -4 + 2\alpha - 2+2\beta_0\right) }, \end{multline*}
where $2 \leq p_{\dec} \leq 6$. By taking $p_{\dec} = 4$, this simplifies to
\[ \int_0^{2\pi} \int \left\lvert \pi_{\theta \#}\mu_g\right\rvert^2 \, d\mathcal{H}^2 \, d\theta
 \lesssim \mu(\mathbb{R}^3) c_{\alpha}(\mu) R^{O(\delta)} \max\left\{ 1, R^{\frac{\beta_0}{2} - \frac{(\alpha-2)}{2} } \right\}. \]
By the condition $(1-c) < \frac{\alpha-2}{100\beta_{\inf}} < \frac{\alpha-2}{\beta_{0}}$ and since $p=4/3$, this implies that 
\[ \int_0^{2\pi} \int \left\lvert \pi_{\theta \#}\mu_g\right\rvert^2 \, d\mathcal{H}^2 \, d\theta \leq R^{-\epsilon}  \mu(\mathbb{R}^3) c_{\alpha}(\mu) R^{\beta \left(\frac{2}{p}-1 \right) }, \]
which verifies \eqref{goodsufficient}. \end{proof}

	\begin{proof}[Proof of Theorem~\ref{projection43}] \begin{sloppypar} Let $p = 4/3$. Take $\beta>0$ such that $\beta \ll \alpha-2$, and let $\epsilon = \beta/10^{10}$. Choose $\delta \ll \varepsilon = \varepsilon(\beta)$. Write $\mu = \mu_b + \mu_g$, where the ``bad'' part is defined as in \eqref{badpart}, except that the parameter $\beta$ is used instead of $\beta_0$ in \eqref{badplanks}. By the triangle inequality, and then by following the same argument as in the proof of Lemma~\ref{energylemma}, 
	\begin{multline*} \left(\int \left\lvert \pi_{\theta \#}\mu_b\right\rvert^{p} \, d\mathcal{H}^2\right)^{1/p} \lesssim  \sum_{j \geq 1} \sum_{k \in [j\epsilon, j ] } \Bigg( O(j)\sum_{m} \int_0^{2\pi} \sum_{\substack{\tau \in \Lambda_{j,k}: \\ \left\lvert \theta_{\tau} - \theta \right\rvert \leq 2^{k\left(\delta-1/2\right)}}} \sum_{\substack{ T \in \mathbb{T}_{\tau,b} : \\
T \cap B_m \neq \emptyset } }  \\
 2^{2(j-k)(p-1) + kO(\delta)} \int_{\pi_{\theta}(2^{j-k-k\delta}2T) } \left\lvert \pi_{\theta\#} 2^{j-k-k\delta}_{\#} \mu_{m,j,k}\right\rvert^p \, d\mathcal{H}^2 \, d\theta\Bigg)^{1/p}, \end{multline*}
where each $2^{j-k-k\delta}_{\#} \mu_{m,j,k}$ is essentially constant on balls of radius $2^{-k}$. By Lemma~\ref{energylemma}, 
\begin{equation} \label{badlpbound} \int \left\lvert \pi_{\theta \#}\mu_b\right\rvert^{p} \, d\mathcal{H}^2 \lesssim c_{\alpha}(\mu)^{p-1} \mu(\mathbb{R}^3) \left(\sum_{j \geq 1} 2^{-j \varepsilon}\right)^p \lesssim c_{\alpha}(\mu)^{p-1} \mu(\mathbb{R}^3). \end{equation}
This bounds the ``bad'' part, so it remains to bound the ``good'' part. \end{sloppypar}

 By the argument from Lemma~\ref{energylemma} (or \cite{GGGHMW}), 
\begin{multline*} \int_0^{2\pi} \int \left\lvert \pi_{\theta \#}\mu_g\right\rvert^2 \, d\mathcal{H}^2 \, d\theta 
\lesssim I_{2+100\epsilon}(\mu) \\
+ \mu(\mathbb{R}^3) c_{\alpha}(\mu) \sum_{j \geq 1} \sum_{k \in [\epsilon j, j ] } 2^{-j + k/2+kO(\delta)} 2^{j (3-\alpha) + k \left( \frac{1}{2} - \frac{1}{p_{\dec}} \right) \left( -4 + 2\alpha - 2+2\beta\right) }. \end{multline*}
By taking $p_{\dec} = 4$ and using $\alpha > 2$ and $\beta \ll \alpha-2$, this is a decaying geometric series, which yields 
\[ \int_0^{2\pi} \int \left\lvert \pi_{\theta \#}\mu_g\right\rvert^2 \, d\mathcal{H}^2 \, d\theta \lesssim \mu(\mathbb{R}^3) c_{\alpha}(\mu). \] 
It was shown in \cite{GGGHMW} that $\pi_{\theta\#} \mu \in L^1$ and $\pi_{\theta\#}\mu_b \in L^1$ for a.e.~$\theta \in [0, 2\pi)$, and that
\[ \int_0^{2\pi} \int \left\lvert \pi_{\theta \#}\mu_b\right\rvert \, d\mathcal{H}^2 \, d\theta \lesssim \mu(\mathbb{R}^3).\]
By the triangle inequality, this gives 
\[ \int_0^{2\pi} \int \left\lvert \pi_{\theta \#}\mu_g\right\rvert \, d\mathcal{H}^2 \, d\theta \lesssim \mu(\mathbb{R}^3).\]
By writing $p = (1-\phi) 1 + \phi \cdot 2$ where $\phi = p-1$, and then applying Hölder's inequality to the function 
\[ \left\lvert \pi_{\theta \#}\mu_g\right\rvert^p = \left\lvert \pi_{\theta \#}\mu_g\right\rvert^{(1-\phi) \cdot 1} \cdot \left\lvert \pi_{\theta \#}\mu_g\right\rvert^{\phi \cdot 2}, \]
with $q = \frac{1}{1-\phi}$ and $q' = \frac{1}{\phi}$, this yields 
\begin{equation} \label{goodlpbound} \int_0^{2\pi} \int \left\lvert \pi_{\theta \#}\mu_g\right\rvert^p \, d\mathcal{H}^2 \lesssim c_{\alpha}(\mu)^{p-1} \mu(\mathbb{R}^3). \end{equation}
Combining \eqref{badlpbound} and \eqref{goodlpbound} gives 
\[ \int_0^{2\pi} \int \left\lvert \pi_{\theta \#}\mu\right\rvert^p \, d\mathcal{H}^2 \lesssim c_{\alpha}(\mu)^{p-1} \mu(\mathbb{R}^3); \]
by the triangle inequality. \end{proof}
 
\begin{proof}[Proof of Theorem~\ref{sl2maximal}] It may be assumed that all tubes $T \in \mathbb{T}$ have $\ell(T) = p \ast \mathbb{V}_{\theta}$ with $|\theta| \geq \pi/4$. For each $T \in \mathbb{T}$, let $w_T$ be such that $\ell(w_T)$ is the centre line of $T$. Let $\mu = \frac{1}{\delta}\sum_{T \in \mathbb{T}} \chi_{B(w_T, C\delta)}$, where $C$ is a large constant to be chosen. Then by the point-line duality argument in \eqref{simplify},  
\[ \left\lVert \sum_{T \in \mathbb{T} } \chi_T \right\rVert_{L^{4/3}(B(0,1) ) }^{4/3} \lesssim \int \| \pi_{\theta\#} \mu \|^{4/3}_{4/3} \, d\theta \leq C_{\epsilon} \delta^{-\epsilon} \mu(\mathbb{R}^3) c_2(\mu)^{1/3}, \]
where the constant $C$ is now chosen large enough to reverse the inequality in \eqref{presimplify}. But $\mu(\mathbb{R}^3) \lesssim \delta^2|\mathbb{T}|$, and $c_2(\mu) \lesssim 1$. This gives \eqref{dualkakeya}. \end{proof}

\section{An \texorpdfstring{$L^{6/5}$}{L6/5} inequality for projections onto lines} 
	
		\begin{lemma}  \label{energylemma2} Let $\alpha > 1$ and let $p=6/5$. For any $\beta >0$, there exists an $\varepsilon>0$ (depending on $\alpha$ and $\beta$) such that 
\begin{equation} \label{geometric2} \int_0^{2\pi} \int_{\bigcup_{D \in \mathbb{D}_{\theta}}} \left\lvert \rho_{\theta \#}(\mu \ast \phi_R)\right\rvert^p \, d\mathcal{H}^1 \, d\theta  \leq C_{\alpha,\beta} R^{-\varepsilon} \mu(\mathbb{R}^3) c_{\alpha}(\mu)^{p-1}, \end{equation}
for any $R > 1$, for any Borel measure $\mu$ on $B_3(0,1)$ with $c_{\alpha}(\mu) < \infty$, and for any family of sets $\{ \mathbb{D}_{\theta} \}$, where each $\mathbb{D}_{\theta}$ is a set of intervals of diameter $R^{-1}$ in $\spn(\gamma(\theta) )$, each with cardinality $\lvert\mathbb{D}_{\theta}\rvert \leq R^{1-\beta} \mu(\mathbb{R}^3)c_{\alpha}(\mu)^{-1}$, such that the integrand in \eqref{geometric2} is Borel measurable. \end{lemma} 
\begin{proof} Relabel $\mu$ as $\mu \ast \phi_R$. Let $\beta_{\inf}$ be the infimum over all $\beta > 0$ such that the conclusion of the lemma is true. Clearly $\beta_{\inf} \leq 1$. Suppose for a contradiction that $\beta_{\inf} >0$. Let $\beta_0 \in \left( \beta_{\inf}, 2 \beta_{\inf} \right)$, and let $\delta>0$ be such that $\delta \ll \min\{\beta, \varepsilon(\beta_0)\}$. Let $\beta > c \beta_0$, where $c\in (0,1)$ is chosen sufficiently close to 1 to ensure that $(1-c) < \frac{\alpha-1}{1000 \beta_{\inf}}$. Let $J$ be such that $2^J \sim R^{\beta/10^{10} }$. To obtain a contradiction, it suffices to prove that the conclusion of the lemma holds for $\beta$. Write 
\[ \mu = \mu_b + \mu_g, \]
where 
\begin{equation} \label{mubaddefn} \mu_b := \sum_{j \geq J} \sum_{\tau \in \Lambda_j }  \sum_{T \in \mathbb{T}_{\tau, b } } M_T \mu, \qquad \mu_g := \mu-\mu_b, \end{equation}
\[ M_T \mu = \eta_T \left( \mu \ast \widecheck{\psi_{\tau } } \right), \]
and where $\{ \psi_{\tau} \}_{\tau \in \Lambda}$ is a smooth partition of unity on the $\sim 1$ neighbourhood of the light cone minus the set $B_3(0,1)$, $\Lambda = \bigcup_{j \geq 1} \Lambda_j$, and each $\Lambda_j$ is a boundedly overlapping cover of the $\sim 1$ neighbourhood of the light cone intersected with $B(0, 2^j) \setminus B(0, 2^{j-1} )$, by boxes of dimensions $\sim 1 \times 2^{j/2} \times 2^j$. For each $\tau$, the set $\mathbb{T}_{\tau}$ is a finitely overlapping cover of $\mathbb{R}^3$ by planks $T$ of dimensions $2^{j \delta} \times 2^{-j(1/2-\delta) } \times 2^{-j(1-\delta)}$ dual to $\tau$, and 
\begin{equation} \label{badplanksdefn} \mathbb{T}_{\tau,b} = \left\{ T \in \mathbb{T}_{\tau}: \mu(2T) \geq 2^{-j(1-\beta_0) } c_{\alpha}(\mu)\right\}. \end{equation}
By the triangle inequality and Hölder's inequality,
\begin{multline*} \int_0^{2\pi} \int_{\bigcup_{D \in \mathbb{D}_{\theta}}} \left\lvert \rho_{\theta \#} \mu\right\rvert^p \, d\mathcal{H}^1 \, d\theta \\
\lesssim \int_0^{2\pi} \int \left\lvert \rho_{\theta \#}\mu_b\right\rvert^p \, d\mathcal{H}^1 \, d\theta +  \left( R^{-\beta\left( \frac{2}{p} - 1\right) }\int_0^{2\pi} \int \left\lvert \rho_{\theta \#}\mu_g \right\rvert^2 \, d\mathcal{H}^1 \, d\theta\right)^{p/2}. \end{multline*} 
Similarly to the proof of Lemma~\ref{energylemma}, it may be assumed that negligible tail terms do not dominate the inequalities that follow, so they may be left out. By the triangle inequality followed by Hölder's inequality again,
\begin{equation} \label{doubleasterisk} \int_0^{2\pi} \int \left\lvert \rho_{\theta \#}\mu_b\right\rvert^p \, d\mathcal{H}^1 \, d\theta \lesssim (\log R)^{O(1)} \sum_{j \geq J}   \int_0^{2\pi} \int \left\lvert \sum_{\tau \in \Lambda_j} \sum_{T \in \mathbb{T}_{\tau,b} }  \rho_{\theta \#} M_T \mu \right\rvert^p \, d\mathcal{H}^1 \, d\theta. \end{equation}
Fix $j \geq J$. It will be shown that
\begin{multline} \label{decoupling2}  \int_0^{2\pi}  \int \left\lvert \sum_{\tau \in \Lambda_j} \sum_{T \in \mathbb{T}_{\tau,b} }  \rho_{\theta \#} M_T \mu \right\rvert^p \, d\mathcal{H}^1 \,d \theta  \\
\lesssim 2^{jO(\delta)}  \int_0^{2\pi} \sum_{\substack{\tau \in \Lambda_j: \\ \left\lvert \theta_{\tau} - \theta \right\rvert \leq 2^{j\left(\delta-1/2\right)}}} \sum_{T \in \mathbb{T}_{\tau,b} }  \mu(2T)^p 2^{j(p-1)} \, d\theta, \end{multline}
where $\theta_{\tau}$ is such that $\gamma(\theta_{\tau})$ is the centre line of $\tau$. If $p$ were equal to 2 then \eqref{decoupling2} would be a consequence of the following. For each $\theta$, apply Plancherel's theorem in one dimension to change the left-hand side of \eqref{decoupling2} to an integral over the light-cone  at distance $\sim 2^j$ from the origin. The uncertainty principle, using that $\mu$ is supported in the unit ball, changes this integral to an integral over $\mathbb{R}^3$ (essentially supported in the $\sim 1$ neighbourhood of the light cone in $\mathbb{R}^3$ at distance $\sim 2^j$ from the origin). Plancherel's theorem in $\mathbb{R}^3$ can then be used to decouple the $L^2$ norms (this is similar to the argument from Eq.~2.16 to Eq.~2.18 in \cite{lengthprojections}), and then \eqref{decoupling2} follows from the uncertainty principle (and the same holds if the set of planks in the left-hand side of \eqref{decoupling2} is replaced by any subcollection). If $p$ were equal to 1 then \eqref{decoupling2} would follow from the triangle inequality and stationary phase (Lemma~2.2 of \cite{lengthprojections}). This implies \eqref{decoupling2} if either $p=1$ or $p=2$. The case $p = 6/5$ for \eqref{decoupling2} holds by interpolation; by pigeonholing it may be assumed that the sum over the $T$'s in the left-hand side of \eqref{decoupling2} is over a subset of the $T$'s such that $\mu(2T)$ is constant up to a factor of 2, and then \eqref{decoupling2} follows writing $p = (1-\phi) \cdot 1 + \phi \cdot 2$ and applying Hölder's inequality with $q = 1/(1-\phi)$ and $q' = 1/\phi$. This proves \eqref{decoupling2}.

 By Hölder's inequality applied to the right-hand side of \eqref{decoupling2}, substituted into \eqref{doubleasterisk},
\begin{multline*} \int_0^{2\pi} \int \left\lvert \rho_{\theta \#}\mu_b\right\rvert^p \, d\mathcal{H}^1 \, d\theta \lesssim (\log R)^{O(1)} \sum_{j \geq J}  2^{j O(\delta)} \\
 \int_0^{2\pi}  \int_{\bigcup_{\substack{\tau \in \Lambda_j, T \in \mathbb{T}_{\tau, b } \\
|\theta_{\tau} - \theta| \leq 2^{-j(1/2-\delta) }}} \rho_{\theta}(2T)  } \left\lvert  \rho_{\theta \#} \left(\mu \ast \phi_{2^{j(1-\delta) } } \right) \right\rvert^p \, d\mathcal{H}^1 \, d\theta. \end{multline*}
By the definition of the ``bad'' part, this is $\lesssim R^{-\varepsilon\beta/10^{10}} \mu(\mathbb{R}^3) c_{\alpha}(\mu)^{p-1}$, so this bounds the ``bad'' part. 

For the ``good'' part, it suffices to prove that 
\[ \int_0^{2\pi} \int \left\lvert \rho_{\theta \#}\mu_g \right\rvert^2 \, d\mathcal{H}^1 \, d\theta \ll R^{\beta\left( \frac{2}{p} - 1\right) } \mu(\mathbb{R}^3) c_{\alpha}(\mu). \]
By similar working to that in \cite{lengthprojections}, it suffices to show that 
\[ \sum_{j \geq J} 2^{j(-1+O(\delta) ) } \sum_{\tau \in \Lambda_j} \sum_{T \in \mathbb{T}_{\tau, g } } \int_{\mathbb{R}^3}  \left\lvert M_T \mu \right\rvert^2 \ll R^{\beta\left( \frac{2}{p} - 1\right)} \mu(\mathbb{R}^3) c_{\alpha}(\mu). \]
Similar working to that in \cite{lengthprojections} yields that, for fixed $j$ with $j \geq J$, 
\[ \sum_{\tau \in \Lambda_j} \sum_{T \in \mathbb{T}_{\tau, g } } \int_{\mathbb{R}^3}  \left\lvert M_T \mu \right\rvert^2 \lesssim 2^{j \left( \frac{5-2\alpha}{p_{\dec} } +\frac{1}{2} + \beta_0 \left( \frac{p_{\dec}-2}{p_{\dec} } \right) \right) + jO(\delta) }\mu(\mathbb{R}^3) c_{\alpha}(\mu), \]
where $2 \leq p_{\dec} \leq 6$. Therefore, it suffices to show that 
\[ R^{\frac{5-2\alpha}{p_{\dec} } -\frac{1}{2} + \beta_0 \left( \frac{p_{\dec}-2}{p_{\dec} } \right)} \ll R^{\beta\left( \frac{2}{p} - 1\right) + O(\delta)}. \]
By taking $p_{\dec} = 6$, and using that $\delta$ is negligible, it suffices to check that 
\[ \frac{-(\alpha-1)}{3} + \frac{2\beta_0}{3} <  \beta\left( \frac{2}{p} - 1 \right). \]
This inequality does hold since $\beta> c \beta_0$ where $(1-c) < \frac{\alpha-1}{1000\beta_{\inf}}$, and $p=6/5$. \end{proof}
	
	\begin{proof}[Proof of Theorem~\ref{projection44}] Let $p = 6/5$. Take $\beta>0$ such that $\beta \ll \alpha-1$, and let $\delta>0$ be such that $\delta \ll \beta$. Write $\mu = \mu_b + \mu_g$, where the bad part is defined as in \eqref{mubaddefn}, except that the parameter $\beta$ is used in place of $\beta_0$ in \eqref{badplanksdefn}. By the triangle inequality, followed by the same argument as in the proof of Lemma~\ref{energylemma2}, 
	\begin{multline*} \left(\int_0^{2\pi} \int \left\lvert \rho_{\theta \#}\mu_b\right\rvert^{p} \, d\mathcal{H}^1 \, d\theta \right)^{1/p} \lesssim  \sum_{j \geq 1} \Bigg( 2^{jO(\delta)}  \\
\int_0^{2\pi} \sum_{\substack{\tau \in \Lambda_j: \\ \left\lvert \theta_{\tau} - \theta \right\rvert \leq 2^{j\left(\delta-1/2\right)}}} \sum_{ T \in \mathbb{T}_{\tau,b} } \int_{\rho_{\theta}(2T) } \left\lvert \rho_{\theta\#} \left(\mu \ast \phi_{2^{j(1-\delta)}} \right)\right\rvert^p \, d\mathcal{H}^1 \, d\theta\Bigg)^{1/p}. \end{multline*}
 By Lemma~\ref{energylemma2}, this yields
\begin{equation} \label{badlpbound2} \int_0^{2\pi} \int \left\lvert \rho_{\theta \#}\mu_b\right\rvert^{p} \, d\mathcal{H}^1 \, d\theta \lesssim c_{\alpha}(\mu)^{p-1} \mu(\mathbb{R}^3) \left(\sum_{j \geq 1} 2^{-j \varepsilon/10^{10}}\right)^p \lesssim c_{\alpha}(\mu)^{p-1} \mu(\mathbb{R}^3). \end{equation}
This bound the ``bad'' part.

It remains to bound the ``good' part. By the argument from Lemma~\ref{energylemma2} (or \cite[p.~11]{lengthprojections}), 
\[ \int_0^{2\pi} \int \left\lvert \rho_{\theta \#}\mu_g\right\rvert^2 \, d\mathcal{H}^1 \, d\theta \lesssim \mu(\mathbb{R}^3) c_{\alpha}(\mu) \sum_{j \geq 1} 2^{j \left( \frac{5-2\alpha}{p_{\dec} } -\frac{1}{2} + \beta \left( \frac{p_{\dec}-2}{p_{\dec} } \right) \right) + j O(\delta)}, \]
where $2 \leq p_{\dec} \leq 6$. By taking $p_{\dec} = 6$ and using $\alpha > 1$ and $\beta \ll \alpha-1$, this is a decaying geometric series, which yields 
\[ \int_0^{2\pi} \int \left\lvert \rho_{\theta \#}\mu_g\right\rvert^2 \, d\mathcal{H}^1 \, d\theta \lesssim \mu(\mathbb{R}^3) c_{\alpha}(\mu). \] 
It was shown in \cite[p.~14]{lengthprojections} that 
\[ \int_0^{2\pi} \int \left\lvert \rho_{\theta \#}\mu_b\right\rvert \, d\mathcal{H}^1 \, d\theta \lesssim \mu(\mathbb{R}^3).\]
By the triangle inequality, this implies that
\[ \int_0^{2\pi} \int \left\lvert \rho_{\theta \#}\mu_g\right\rvert \, d\mathcal{H}^1 \, d\theta \lesssim \mu(\mathbb{R}^3).\]
By writing $p = (1-\phi) 1 + \phi \cdot 2$ where $\phi = p-1$, and by applying Hölder's inequality to the function 
\[ \left\lvert \rho_{\theta \#}\mu_g\right\rvert^p = \left\lvert \rho_{\theta \#}\mu_g\right\rvert^{(1-\phi) \cdot 1} \cdot \left\lvert \rho_{\theta \#}\mu_g\right\rvert^{\phi \cdot 2}, \]
with $q = \frac{1}{1-\phi}$ and $q' = \frac{1}{\phi}$, this yields 
\begin{equation} \label{goodlpbound2} \int_0^{2\pi} \int \left\lvert \rho_{\theta \#}\mu_g\right\rvert^p \, d\mathcal{H}^1 \,d\theta \lesssim c_{\alpha}(\mu)^{p-1} \mu(\mathbb{R}^3). \end{equation}
Combining \eqref{badlpbound2} and \eqref{goodlpbound2} gives 
\[ \int_0^{2\pi} \int \left\lvert \rho_{\theta \#}\mu\right\rvert^p \, d\mathcal{H}^1 \, d\theta \lesssim c_{\alpha}(\mu)^{p-1} \mu(\mathbb{R}^3). \qedhere \] \end{proof}

\section{Some open problems} \begin{enumerate}
\item The first problem to mention is whether 4/3 can be raised to 3/2 in Theorem~\ref{sl2maximal}, or Theorem~\ref{projection43}. 

\item I do not know if the exponent 3/2 would be a reasonable conjectured sharp exponent to replace 6/5 in Theorem~\ref{projection44}, or if there are other counterexamples which show that the conjectured exponent should be lower. 

\item Another problem is whether $B(0,1)$ in Theorem~\ref{sl2maximal} can be replaced by $\mathbb{R}^3$; unlike for the standard Kakeya maximal function, the local and global versions do not seem to be obviously equivalent, since horizontal or $SL_2$ lines are not preserved by Euclidean translations. 

\item An open problem that lies in-between the stated intersection and projection theorems, is whether there is a ``restricted'' Besicovitch-Federer projection theorem in $\mathbb{R}^3$. \end{enumerate}

\end{document}